\documentclass[11pt,reqno]{article}
\usepackage{bbding}
\usepackage{pifont}
\usepackage{amsmath}
\usepackage{mathrsfs}
\usepackage{amsfonts,bm}
\usepackage{amsthm}
\usepackage{epsfig}
\usepackage[]{harpoon}
\usepackage[active]{srcltx}
\usepackage{indentfirst,latexsym}
\usepackage{psfrag}
\usepackage[]{amssymb}
\usepackage{cite}

\newcommand{\new}{\newcommand*}\new{\rnew}{\renewcommand*}
\new{\newe}{\newenvironment*}\new{\stl}{\setlength}
\stl{\textwidth}{155mm}\stl{\textheight}{22cm}\stl{\headheight}{0cm}
\stl{\topmargin}{0cm}\stl{\oddsidemargin}{0.5cm}\stl{\evensidemargin}{0cm}
\rnew{\arraystretch}{1.2}\rnew{\baselinestretch}{1.2}
\renewcommand{\thefootnote}{\ding{73}}
\newtheorem{thm}{Theorem}[section]

\newtheorem{lem}{Lemma}[section]

\newtheorem{defn}{Definition}[section]

\newtheorem{prob}{Problem}[section]

\newcommand{\eps}{\varepsilon}

\newcommand{\fr}{\frac}

\newcommand{\pa}{\partial}
\newcommand{\arccot}{\mathrm{arccot}}

\numberwithin{equation}{section}

\newe{keywords}
   {\begin{quote}{\bf Keywords:}}
      {\end{quote}}
\newe{AMS}
   {\begin{quote}{\bf AMS subject classification 2000:}}
      {\end{quote}}
\newe{MSC}{\vspace*{5mm}
    {\noindent\it Mathematics Subject Classification(2000):}}{}
\new{\sect}[1]{\section{#1}\setcounter{equation}{0}
 \setcounter{thm}{0}\setcounter{lmm}{0}\setcounter{rmk}{0} }

\begin{document}

\title{ On a global supersonic-sonic patch characterized by  2-D steady full Euler equations  }

\author{
Yanbo Hu$^{a}$ and Jiequan Li$^{b,c,}$\footnote{Jiequan LI is supported by NSFC (nos. 11771054, 91852207) and Foundation of LCP.}
\\{\small \it $^a$Department of Mathematics, Hangzhou Normal University,
Hangzhou, 311121, PR China}
\\
{\small \it $^b$Institute of Applied Physics and Computational Mathematics, Beijing, 100088, PR China}\\
{\small\it $^c$Center for Applied Physics and Technology, Peking University, 100084, PR China}
}

\rnew{\thefootnote}{\fnsymbol{footnote}}

\footnotetext{ Email address: yanbo.hu@hotmail.com (Y. Hu), li\_jiequan@iapcm.ac.cn (J. Li). }

\date{}

\maketitle
\begin{abstract}

Supersonic-sonic patches are ubiquitous in regions of transonic flows and they boil down to a family of degenerate hyperbolic problems in regions  surrounded by a streamline, a characteristic curve and a possible sonic curve. This paper  establishes  the global existence of solutions in a whole supersonic-sonic patch characterized by the two-dimensional full system of steady Euler equations and studies solution behaviors near sonic curves, depending on the proper choice of boundary data extracted from the airfoil problem and related contexts. New characteristic decompositions  are developed  for the full system and a delicate local partial hodograph transformation is introduced for the solution estimates.  It is shown that the solution is uniformly $C^{1,\frac{1}{6}}$ continuous up to the sonic curve and the sonic curve is also $C^{1,\frac{1}{6}}$ continuous.
\end{abstract}

\begin{keywords}
Full Euler equations, transonic flow, sonic curve, supersonic-sonic patches, characteristic decomposition, partial hodograph transformation.
\end{keywords}

\begin{AMS}
35L65, 35L80, 76H05.
\end{AMS}

\section{Introduction}\label{S1}

Supersonic-sonic patches are ubiquitous in regions of transonic flows, just as described in the famous book  (Supersonic Flow and Shock Waves, 1948, Page 370, \cite{Courant}):  {\em Suppose the duct walls are plane except for a small inward bulge at some section. If the entrance Mach number is not much below the value one, the flow becomes supersonic in a finite region adjacent to the bulge and is again purely subsonic throughout the exit section.}  See Fig. \ref{fig1} for the illustration of a flow over an airfoil in gas dynamics.
\vspace{0.2cm}

\begin{figure}[htbp]
\begin{center}
\includegraphics[scale=0.5]{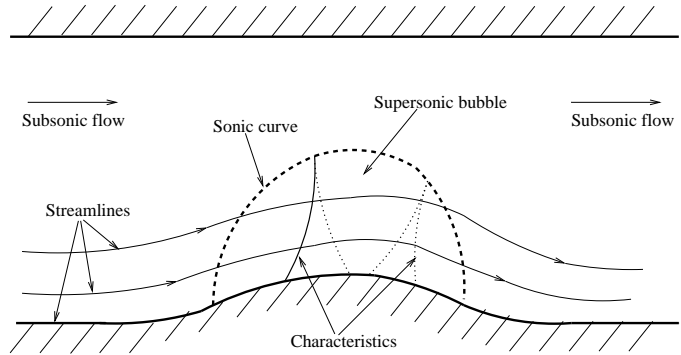}
\caption{\footnotesize Transonic phenomena in a duct.}\label{fig1}
\end{center}
\end{figure}

The existence of  solutions for such a  transonic flow problem  has been extensively studied but still remains open  in a ``global" transonic sense mathematically. See the review paper \cite{Chen-S}.   For mathematical simplicity and manipulation, many works were contributed to understand the transonic structures characterized by  the two-dimensional steady isentropic irrotational compressible Euler equations. See, e.g., \cite{Chen1, Ferrari, Guderley, Morawetz1, Morawetz2, Shiffman}.   As early as the 1950s, Morawetz \cite{Morawetz1} showed the nonexistence of smooth solutions for the transonic problem in general, but  the existence and stability of multidimensional transonic potential flows through an infinite nozzle could be established in \cite{Chen2}. From the subsonic side, Xie and Xin studied the existence of global solutions in a subsonic-sonic part of the nozzle in \cite{Xie-Xin1}, and further verified the well-posedness for the subsonic and subsonic-sonic flows with critical mass flux in \cite{Xie-Xin2}. Chen, Huang and Wang \cite{Chen3} investigated the global existence of subsonic-sonic flows for the multidimensional full Euler equations in the compensated-compactness framework.
The study of transonic shocks arising in supersonic flow past a blunt body or a bounded nozzle was presented among others in \cite{ChenS, Elling-Liu, Xin-Yin1, Xin-Yin2}.

As for supersonic-sonic parts,  there were also many contributions, e.g. in \cite{CZ, LYZ, Li-Zheng2, Sheng-You} and references therein,  adopting the characteristic decomposition method in \cite{LZZ} for the irrotational system of isentropic Euler equations.  The supersonic-sonic  patches, also named as  semi-hyperbolic patches, often appear in various contexts of the two-dimensional Riemann problem \cite{Glimm, L-Z-Y, Zheng01},  Guderley shock reflection \cite{Tesdall} and  the transonic flow described \cite{Cole-Cook, Courant}. In order to describe this type of patches,   Song and Zheng \cite{Song-Zheng}  first used  the pressure-gradient system. Then  their result was extended using  the isentropic/isothermal Euler equations \cite{Lim, Hu-Li-Sheng} and the related system \cite{Hu-Wang}. The regularity of semi-hyperbolic patch problems was discussed in \cite{Wang-Zheng} for the pressure-gradient system and in \cite{SWZ, Hu-LiT} for the isentropic Euler equations. It is worthwhile to mention  the work of Lai and Sheng \cite{Lai-Sheng} about a  centered wave bubble with sonic boundary.

\begin{figure}[htbp]
\begin{center}
\includegraphics[scale=0.6]{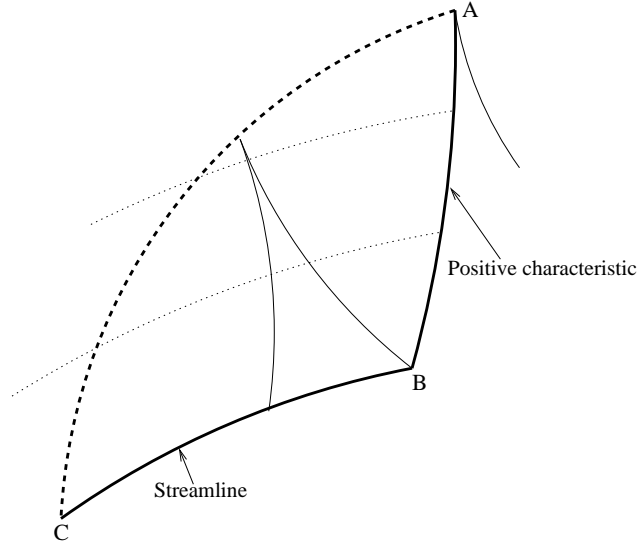}
\caption{\footnotesize The region $ABC$.}\label{fig2}
\end{center}
\end{figure}

Due to the presence of  transonic shocks possibly in the transonic flow,   the entropy in the flow is not uniform and the flow behind the shock is not irrotational \cite{Majda}. Hence  it is more suitable to adopt the full system of  Euler equations to characterize the  corresponding flows,
\begin{align}\label{1.1}
\left\{
\begin{array}{l}
  (\rho u)_x+(\rho v)_y=0, \\
  (\rho u^2+p)_x+(\rho uv)_y=0, \\
  (\rho uv)_x+(\rho v^2+p)_y=0, \\
  (\rho Eu+pu)_x+(\rho Ev+pv)_y=0,
\end{array}
\right.
\end{align}
where $\rho$, $(u, v)$, $p$ and $E$ are, respectively, the density, the velocity, the pressure and the specific total energy.  For  polytropic gases,  $E=\fr{u^2+v^2}{2}+\frac{1}{\gamma-1}\frac{p}{\rho}$,  where $\gamma>1$ is the adiabatic gas constant. The supersonic-sonic patch, as illustrated in Fig. \ref{fig2},  is extracted from Fig. \ref{fig1} near the sonic curve.  We denote the dashed arc  $\widehat{AC} $ as a sonic curve on which the sound speed $c$ is identical with the flow velocity $q=\sqrt{u^2+v^2}$, the solid segment $\widehat{BC}$ as a streamline, and the solid segment $\widehat{AB}$ as a characteristic curve.  Our supersonic-sonic patch problem is formulated as follows.

\begin{prob}\label{p1}
Given two smooth curves $\widehat{CB}$
and $\widehat{BA}$, we prescribe  the supersonic  boundary data on $\widehat{CB}$ and $\widehat{BA}$ such that $q>c$ on $\{\widehat{CB}\cup\widehat{BA}\}\setminus\{C, A\}$ and $q=c$ at the points $C$ and $A$. Moreover, $\widehat{CB}$ is a streamline and $\widehat{BA}$ is a characteristic curve. We seek a smooth solution for the system \eqref{1.1} in the sector region $ABC$ up to the sonic boundary $\widehat{AC}$.  See Fig. \ref{fig2}.
\end{prob}

As pointed out before, this  patch  can be regarded as a supersonic bubble in the transonic flow problem in Fig. \ref{fig1}. For the full Euler system \eqref{1.1}, the streamlines  may be tangent to level sets of Mach number in the region $ABC$ so as to make the problem quite complicated. In the present paper, we consider a special case of \eqref{1.1},  which still reflects the special role of entropy and vorticity, and establish the  existence of smooth solutions in the whole sector region $ABC$. Moreover, we verify that the solution is uniformly $C^{1,\fr{1}{6}}$ continuous up to the sonic curve $\widehat{CA}$ and the curve $\widehat{CA}$ is $C^{1,\fr{1}{6}}$ continuous too.  We just remind that this ``global" solution is built on  a local classical sonic-supersonic solution to system \eqref{1.1}, that was constructed in a recent work \cite{Hu-Li}.  This patch, as a ladder step we believe, is a useful part in the future construction of  complete transonic flow structures.

There are at least three novelties in this paper. First, it is a first attempt to use the full system of  Euler equations to characterize such a supersonic-sonic patch, in which the  entropy and vorticity play a special role. Second,  we need to develop new mathematical techniques to
deal with such a problem, such as new characteristic decompositions and locally partial hodograph transformation using the Mach angle and the inclination angle of streamline, which are substantially different from those in literature. For example,
 in \cite{Kuzmin} Kuz'min established
an existence theorem of a transonic perturbation problem by employing the stream-potential  coordinates. Song, Wang and Zheng \cite{SWZ} used the function $\sqrt{q^2-c^2}$ and the pseudo-potential function as the coordinate system to discuss the regularity of semi-hyperbolic patch problems characterized by the isentropic Euler equations. See \cite{ZhangT1, ZhangT2} for other applications of coordinate transformations near sonic curves. We note that the above auxiliary  coordinate systems can not be applied for the full Euler system \eqref{1.1} due to  the non-existence of potential function. Third,  as far as the theory of partial differential equations is concerned,   hyperbolic problems with some degeneracy  are highly interesting and no general theory is available.  The present paper can be regarded as  the meaningful trial beyond two-equation systems.

The rest of the paper is organized as follows. In Section \ref{S2}, we introduce a set of new dependent variables to reset  the problem in terms of  new coordinates and state the main result.  In Section \ref{S3},  we establish the global existence of smooth solutions to this supersonic-sonic patch up to the sonic curve, which is not known a priori.  In Section \ref{S4}, we provide the uniform regularity of solutions up to the sonic boundary and the regularity of sonic curve.

\section{Formulation of supersonic-sonic patch problem and main result}\label{S2}

This section serves to formulate the problem about  the supersonic-sonic patch characterized by the full system of Euler equations \eqref{1.1}. The well-posedness of the problem depends on the proper choice of boundary data.  In order to state the main result clearly, we follow \cite{Li-Zheng1} to introduce the Mach angles, the flow angles, the entropy and the Bernoulli quantity as
dependent variables to rewrite the governing equations.  Then we provide  new characteristic decompositions for later a priori estimates of the solutions inside the patch.

\subsection{Preliminary characteristic decompositions for full Euler equations}

We assume that the flow is  smooth. Then  the full Euler system \eqref{1.1} is written as
\begin{align}\label{2.1}
\textbf{A}\textbf{W}_x+\textbf{B}\textbf{W}_y=0, \ \ \ \ \ \ \ \ \ \ \ \ \ \ \ \ \ \ \ \ \ \ \ \ \ \\[3mm]
\textbf{A}=\left(
 \begin{array}{cccc}
    u & \rho & 0 & 0 \\
    0 & u & 0 & \frac{1}{\rho} \\
    0 & 0 & u & 0 \\
    0 & \gamma p & 0 & u
 \end{array}
\right),
\quad \textbf{B}=
\left(
 \begin{array}{cccc}
    v & 0 & \rho & 0 \\
    0 & v & 0 & 0 \\
    0 & 0 & v & \frac{1}{\rho} \\
    0 & 0 & \gamma p & v
 \end{array}
\right), \quad
\textbf{W}=\left(
 \begin{array}{c}
    \rho \\
    u \\
    v \\
    p
 \end{array}
\right).
\end{align}
Since we intend to investigate supersonic-sonic patches, characteristics  are very  important.  Therefore we need to define eigenvalues and the associated eigenvectors. The eigenvalues of \eqref{2.1} are
\begin{align}\label{1.2}
\Lambda_0=\Lambda_1=\frac{v}{u},\ \ \Lambda_{\pm}=\frac{uv\pm c\sqrt{q^2-c^2}}{u^2-c^2},
\end{align}
where $c=\sqrt{\gamma p/\rho}$ is the speed of sound and $q=\sqrt{u^2+v^2}$ denotes the flow speed. From the expressions of $\Lambda_\pm$, it is evident that the flow may be transonic (of mixed-type): supersonic (hyperbolic) for $q>c$, subsonic (elliptic) for $q<c$ and sonic (parabolically  degenerate) for $q=c$.
The set of points on which $c=q$ is called the {\em sonic curve.}
The four associated  (left) eigenvectors are
$$
\ell_{0}=(0,u,v,0),\quad \ell_{1}=(c^2,0,0,-1),\quad \ell_\pm=(0,-\Lambda_\pm\gamma p,\gamma p,\Lambda_\pm u-v).
$$
 Standard calculation provides  the characteristic form of \eqref{2.1},
\begin{align}\label{2.2}
\left\{
\begin{array}{l}
  uS_x+vS_y=0, \\
  uB_x+vB_y=0, \\
  -c\rho vu_x+c\rho uv_x\pm\sqrt{u^2+v^2-c^2}p_x \\
  \ \ \ \ +\Lambda_\pm(-c\rho vu_y+c\rho uv_y\pm\sqrt{u^2+v^2-c^2}p_y)=0,
\end{array}
\right.
\end{align}
where $S=p\rho^{-\gamma}$ is the entropy function and $B=\fr{u^2+v^2}{2}+\fr{c^2}{\gamma-1}$ is the Bernoulli function.

We introduce the flow angle function and the Mach angle function
\begin{align}\label{2.3}
\tan\theta=\fr{v}{u},\quad \sin\omega=\fr{c}{q}.
\end{align}
Denote
\begin{align}\label{2.4}
\alpha:=\theta+\omega,\quad \beta=\theta-\omega.
\end{align}
Obviously $\alpha$ and $\beta$ are the inclination angles of characteristics,
\begin{align}\label{2.5}
\tan\alpha=\Lambda_+,\quad \tan\beta=\Lambda_-,\quad \tan\theta=\Lambda_0=\Lambda_1.
\end{align}
Furthermore, we adopt the following normalized directional derivatives along the characteristics, as in \cite{LZZ},
\begin{align}\label{2.6}
\begin{array}{l}
\bar{\pa}^+=\cos\alpha\pa_x+\sin\alpha\pa_y,\quad
\bar{\pa}^-=\cos\beta\pa_x+\sin\beta\pa_y, \quad
\bar{\pa}^0=\cos\theta\pa_x+\sin\theta\pa_y,
\end{array}
\end{align}
or conversely,  one has
\begin{align}\label{2.7}
\pa_x=-\fr{\sin\beta\bar{\pa}^+-\sin\alpha\bar{\pa}^-}{\sin(2\omega)},\quad \pa_y=\fr{\cos\beta\bar{\pa}^+-\cos\alpha\bar{\pa}^-}{\sin(2\omega)},\quad
\bar{\pa}^0=\fr{\bar{\pa}^++\bar{\pa}^-}{2\cos\omega}.
\end{align}
Then we can write \eqref{2.2} into a new system in terms of the variables $(S, B, \theta, \omega)$
\begin{align}\label{2.8}
\left\{
\begin{array}{l}
   \bar{\pa}^0S=0, \\
   \bar{\pa}^0B=0, \\
   \bar{\pa}^+\theta+\fr{\cos^2\omega}{\sin^2\omega+\kappa}\bar{\pa}^+\omega=
\fr{\sin(2\omega)}{4\kappa}\bigg(\fr{1}{\gamma}\bar{\pa}^+\ln S-\bar{\pa}^+\ln B\bigg), \\
   \bar{\pa}^-\theta-\fr{\cos^2\omega}{\sin^2\omega+\kappa}\bar{\pa}^-\omega=-
\fr{\sin(2\omega)}{4\kappa}\bigg(\fr{1}{\gamma}\bar{\pa}^-\ln S-\bar{\pa}^-\ln B\bigg).
  \end{array}
\right.
\end{align}
Here $\kappa=(\gamma-1)/2$. We denote
$$
\Omega=\fr{1}{4\kappa}\bigg(\fr{1}{\gamma}\ln S-\ln B\bigg),
$$
and then obtain a subsystem of \eqref{2.8}
\begin{align}\label{2.9}
\left\{
\begin{array}{l}
   \bar{\pa}^0\Omega=0, \\
   \bar{\pa}^+\theta+\fr{\cos^2\omega}{\sin^2\omega+\kappa}\bar{\pa}^+\omega=\sin(2\omega)\bar{\pa}^+\Omega, \\
   \bar{\pa}^-\theta-\fr{\cos^2\omega}{\sin^2\omega+\kappa}\bar{\pa}^-\omega=-\sin(2\omega)\bar{\pa}^-\Omega.
  \end{array}
\right.
\end{align}

Now we need more interpretation for the quantity $\Omega$.
According to the commutator relation between $\bar{\pa}^0$ and $\bar{\pa}^+$ in the previous paper \cite{Hu-Li},
\begin{align}\label{2.10}
\bar{\partial}^0\bar{\partial}^+-\bar{\partial}^+\bar{\partial}^0=
\frac{\cos\omega\bar{\partial}^+\theta-
\bar{\partial}^0\alpha}{\sin\omega}\bar{\partial}^0-
\fr{\bar{\partial}^+\theta-\cos\omega\bar{\partial}^0\alpha}{\sin\omega}\bar{\partial}^+,
\end{align}
we find by \eqref{2.7} and \eqref{2.8} that, for any smooth function $I$ satisfying $\bar{\pa}^0I\equiv0$,
\begin{align}\label{2.11}
\bar{\partial}^0\bar{\partial}^+I=-\frac{\bar{\partial}^+\theta-
\cos\omega\bar{\partial}^0\alpha}{\sin\omega}\bar{\partial}^+I =\frac{(\kappa+1)\cot\omega}{\kappa+\sin^2\omega}\bar{\partial}^0\omega\bar{\partial}^+I,
\end{align}
which imply
\begin{align*}
\bar{\partial}^0\bigg(\fr{\bar{\partial}^+I}{G(\omega)}\bigg)=0,\qquad G(\omega)=\bigg(\fr{\sin^2\omega}{\kappa+\sin^2\omega}\bigg)^{\fr{\kappa+1}{2\kappa}}.
\end{align*}
Then we have
\begin{align}\label{2.12}
\bar{\partial}^0H=0,\quad H=\fr{\bar{\partial}^+\Omega}{G(\omega)}.
\end{align}
Thus we obtain a new system in terms of the variables $(H, \theta, \omega)$
\begin{align}\label{2.9a}
\left\{
\begin{array}{l}
   \bar{\pa}^0H=0, \\
   \bar{\pa}^+\theta+\fr{\cos^2\omega}{\sin^2\omega+\kappa}\bar{\pa}^+\omega=\sin(2\omega)G(\omega)H, \\
   \bar{\pa}^-\theta-\fr{\cos^2\omega}{\sin^2\omega+\kappa}\bar{\pa}^-\omega=\sin(2\omega)G(\omega)H.
  \end{array}
\right.
\end{align}
Here we have used the fact $\bar{\pa}^-\Omega=-\bar{\pa}^+\Omega$, which follows from the equation $\bar{\pa}^0\Omega=0$. In this paper, we consider the case $H\equiv H_0$ in the region $ABC$, which just needs that $H\equiv H_0$ holds on the boundary $\widehat{BA}$ by the first equation of \eqref{2.9a}, where $H_0$ is a nonnegative constant. In particular, the case $H_0=0$ corresponds to the isentropic irrotational flows.
For the above specified boundary data, we only need to consider the following system
\begin{align}\label{2.9b}
\left\{
\begin{array}{l}
   \bar{\pa}^+\theta+\fr{\cos^2\omega}{\sin^2\omega+\kappa}\bar{\pa}^+\omega=H_0\sin(2\omega)G(\omega), \\
   \bar{\pa}^-\theta-\fr{\cos^2\omega}{\sin^2\omega+\kappa}\bar{\pa}^-\omega=H_0\sin(2\omega)G(\omega).
  \end{array}
\right.
\end{align}

In order to establish a priori estimates of solutions, we derive the characteristic decompositions for the angle variables. Introduce a new variable
\begin{align}\label{2.13}
\Xi=\fr{1}{4\kappa}\ln\bigg(\fr{\sin^2\omega}{\kappa+\sin^2\omega}\bigg)-\Omega.
\end{align}
It is easy to check from the last two equations of \eqref{2.9} that
\begin{align}\label{2.14}
\left\{
    \begin{array}{l}
       \bar{\pa}^+\theta+\sin(2\omega)\bar{\pa}^+\Xi=0, \\
       \bar{\pa}^-\theta-\sin(2\omega)\bar{\pa}^-\Xi=0.
    \end{array}
\right.
\end{align}
In addition, one obtains the relations between $\bar\pa^\pm\Xi$ and $\bar\pa^\pm\omega$
\begin{align}\label{2.14a}
\bar\pa^\pm\omega=\fr{2\sin\omega(\kappa+\sin^2\omega)}{\cos\omega}[\bar\pa^\pm\Xi\pm HG(\omega)].
\end{align}
By performing a direct calculation, we acquire the characteristic decomposition for $\Xi$
\begin{align}\label{2.15}
\left\{
\begin{array}{l}
\bar{\pa}^-\bar{\pa}^+\Xi =\fr{\kappa\bar{\pa}^+\Xi+(\kappa+\sin^2\omega)GH}{\cos^2\omega}[\bar{\pa}^+\Xi -\cos(2\omega)\bar{\pa}^-\Xi]+\fr{\bar{\pa}^+\Xi}{\cos^2\omega}[\bar{\pa}^+\Xi+\cos^2(2\omega)\bar{\pa}^-\Xi], \\[6pt]
\bar{\pa}^+\bar{\pa}^-\Xi =\fr{\kappa\bar{\pa}^-\Xi-(\kappa+\sin^2\omega)GH}{\cos^2\omega}[\bar{\pa}^-\Xi -\cos(2\omega)\bar{\pa}^+\Xi]+\fr{\bar{\pa}^-\Xi}{\cos^2\omega}[\bar{\pa}^-\Xi+\cos^2(2\omega)\bar{\pa}^+\Xi].
\end{array}
\right.
\end{align}
For $\theta$, we   have
\begin{align}\label{2.16}
\left\{
\begin{array}{l}
   \bar{\partial}^-\bar{\partial}^+\theta+\bigg\{\frac{\bar{\partial}^+\beta-
\cos(2\omega)\bar{\partial}^-\alpha}{\sin(2\omega)}+
\frac{\cos(2\omega)(\kappa+\sin^2\omega)GH}{\cos^2\omega}\bigg\}\bar{\partial}^+\theta
=\frac{\cos(2\omega)(\kappa+\sin^2\omega)GH}{\cos^2\omega}\bar{\partial}^-\theta, \\[10pt]
   \bar{\partial}^+\bar{\partial}^-\theta+\bigg\{\frac{\cos(2\omega)\bar{\partial}^+\beta-
\bar{\partial}^-\alpha}{\sin(2\omega)}-
\frac{\cos(2\omega)(\kappa+\sin^2\omega)GH}{\cos^2\omega}\bigg\}\bar{\partial}^-\theta
=-\frac{\cos(2\omega)(\kappa+\sin^2\omega)GH}{\cos^2\omega}\bar{\partial}^+\theta.
  \end{array}
\right.
\end{align}
All these characteristic decompositions are necessary in the understanding of solutions in the supersonic-sonic patch, to be constructed in the following sections.

\subsection{Supersonic-sonic boundary data and statement of  main results}

We  prescribe supersonic-sonic boundary conditions for Problem \ref{p1}, mimicking the real setting of airfoil problem.
 Let $\widehat{CB}: y=\varphi(x), x\in[x_1,x_2],$ be a smooth curve. We assume that the curve $\widehat{CB}$ and the boundary values
$(H, \theta, \omega )|_{\widehat{CB}}=(\hat{H}, \hat{\theta}, \hat{\omega})(x)$ satisfy
\begin{align}\label{2.17}
\begin{array}{l}
\varphi(x)\in C^2([x_1,x_2])\cap C^3((x_1,x_2]),\  (\hat{H}, \hat{\theta}, \sin\hat{\omega})(x)\in C^1([x_1,x_2])\cap C^2((x_1,x_2]), \\
\varphi'(x)>0,\quad  \varphi''(x)<0,\quad  0<\varphi_0\leq\varphi'(x)\leq\varphi_1,\\
  \hat{H}(x)\equiv H_0\geq0,\quad \hat{\omega}(x_1)=\fr{\pi}{2}, \quad \hat{\omega}(x_2)\geq\fr{\pi}{4}, \qquad   \forall\ x\in[x_1,x_2],\\
(\sin\hat{\omega}(x))'<0, \quad \hat{\theta}(x)=\arctan\varphi'(x),
\end{array}
\end{align}
where $\varphi_0, \varphi_1$ and $H_0$ are constants. From \eqref{2.17}, we see that the curve $\widehat{CB}$ is an increasing and concave streamline, along which the flow angle $\theta$ and Mach angle $\omega$ are decreasing functions.
The assumption $\hat{\omega}(x_1)=\pi/2$ means that $C$ is a sonic point. Moreover,
we require the following inequality to be satisfied on $\widehat{CB}$
\begin{align}\label{2.19}
\bigg(\fr{\varphi''}{1+(\varphi')^2} -\fr{\cos\hat{\omega}(\sin\hat{\omega})'}{\kappa+\sin^2\hat{\omega}}\bigg)(x)<0,\quad \forall\ x\in[x_1,x_2],
\end{align}
which obviously holds when $\hat{\omega}$ is close to $\pi/2$.

Let $\widehat{BA}: x=\psi(y)\ (y\in[y_2,y_3])$ be a smooth curve satisfying $x_2=\psi(y_2)$.  We assume that the curve $\widehat{BA}$ and the boundary values
$(H, \theta, \omega )|_{\widehat{BA}}=(\tilde{H}, \tilde{\theta}, \tilde{\omega})(y)$ satisfy
\begin{align}\label{2.20}
\begin{array}{l}
\psi(y)\in C^2([y_2,y_3])\cap C^3([y_2,y_3)),\  (\tilde{H}, \tilde{\theta}, \sin\tilde{\omega})(y)\in C^1([y_2,y_3])\cap C^2([y_2,y_3)), \\
|\psi'(y)|\leq \psi_1,\  \psi''(y)<0, \ \tilde{H}(y)=H_0,\  (\tilde{\theta}, \tilde{\omega})(y_2)=(\hat\theta, \hat\omega)(x_2),\\
  \tilde{\theta}(y)+\tilde{\omega}(y)=\arccot\psi'(y), \quad  \tilde{\omega}(y_3)=\fr{\pi}{2},\quad \qquad    \forall\ y\in[y_2,y_3], \\
\tilde{\theta}'(y)+\fr{\cos\tilde\omega(\sin\tilde\omega)'}{\kappa+\sin^2\tilde\omega}(y) =\sin(2\tilde\omega)G(\tilde\omega)H_0,
\end{array}
\end{align}
where $\psi_1$ is a positive constant.
The conditions in the third of \eqref{2.20} mean that the curve $\widehat{BA}$ is a $\Lambda_+$-characteristic and the point $A$ is sonic. The last condition in \eqref{2.20} is the compatibility condition with system \eqref{2.9}.
Furthermore, we require that the functions $\psi(y)$, $\varphi(x)$, $(\hat{\theta},\hat{\omega})(x)$ and $(\tilde\theta,\tilde{\omega})(y)$ satisfy
\begin{align}\label{2.21}
H_0\sqrt{1+(\psi')^2}G(\tilde{\omega})(y)<\fr{(\sin\tilde\omega)'} {2\sin\tilde\omega(\kappa+\sin^2\tilde\omega)}(y), \quad \forall\ y\in[y_2,y_3),
\end{align}
and
\begin{align}\label{2.22}
\fr{\tilde\theta'}{\sqrt{1+(\psi')^2}}(y_2)=\cos\hat\theta\cos\hat\omega \bigg(\fr{\varphi''}{1+(\varphi')^2}-\fr{\cos\hat\omega(\sin\hat\omega)'}{\kappa+\sin^2\hat\omega}\bigg)(x_2).
\end{align}
The condition in \eqref{2.21} is to ensure that the flow angle $\theta$ is decreasing along $\widehat{BA}$. The condition \eqref{2.22} is the compatibility condition at the point $B$.

For the conditions \eqref{2.17}-\eqref{2.22}, it  looks that there are   many assumptions on the boundaries at  the first  glimpse. However, they are all reasonable if one carefully inspects the airfoil problem. In addition, in order to ensure $\tilde{H}(y)=H_0$, it is only necessary to specify the boundary data of $\tilde\Omega$ or $(\tilde{S}, \tilde{B})$ on $\widehat{BA}$ to satisfy
$$
\tilde\Omega'(y)=\fr{1}{4\kappa\gamma}\bigg(\ln\fr{\tilde{S}}{\tilde{B}^\gamma}\bigg)'(y) =H_0\sqrt{1+(\psi'(y))^2}G(\tilde{\omega}(y)), \quad \forall\ y\in[y_2,y_3].
$$

The main result of this paper can be stated as follows.
\begin{thm}\label{thm1}
Assume that the boundary conditions \eqref{2.17}-\eqref{2.22} hold. Then system \eqref{2.9a} with the boundary data $(H, \theta, \omega )|_{\widehat{CB}}=(H_0, \hat{\theta}, \hat{\omega})(x)$ and $(H, \theta, \omega )|_{\widehat{BA}}=(H_0, \tilde{\theta}, \tilde{\omega})(y)$
admits a global smooth solution $(H_0, \theta, \omega)\in C^2$ in the region $ABC$ with the sonic boundary $\widehat{AC}$. Moreover, $(\theta, \sin\omega)(x,y)$ is uniformly $C^{1,\fr{1}{6}}$ up to the sonic boundary $\widehat{AC}$ and the sonic curve $\widehat{AC}$ is $C^{1,\fr{1}{6}}$-continuous.
\end{thm}

\section{Global existence of solutions in the whole patch}\label{S3}

In this section, we use characteristic decompositions as main ingredients to show the existence result in Theorem \ref{thm1}. In particular, the a priori estimates of solutions are established thanks to the non-homogeneous characteristic decompositions \eqref{2.15} and \eqref{2.16}.

\subsection{Boundary data estimates and the local existence}

Before constructing the global solution, we need to inspect the  boundary values on $\widehat{BC}$ and $\widehat{BA}$.  Of course, the global solution is extended from the local existence result at the point $B$, which is also stated at first.
\begin{lem}\label{lem1}
With the assumptions in Theorem \ref{thm1},  we have
\begin{align}\label{3.1}
\begin{array}{ll}
 \bar\pa^0\theta<0,\ \bar\pa^0\sin\omega<0,\ \bar\pa^\pm\theta<0, \ \ \ & \mbox{on }\ \widehat{BC}; \\\
 \bar\pa^+\theta<0,\ \bar\pa^+\sin\omega>0, &\mbox{on } \ \widehat{BA}.
\end{array}
\end{align}
Hereafter the symbols of curves $\widehat{BC}$ and $\widehat{BA}$ do not contain the points $C$ and $A$, respectively.
\end{lem}
\begin{proof}
Recalling the definition of $\bar\pa^0$ in \eqref{2.6}, we obtain $\bar\pa^0=\cos\hat\theta\fr{\rm d}{{\rm d}x}$ on $\widehat{BC}$. Thus one has
\begin{align*}
\bar\pa^0\theta|_{\widehat{BC}}=\cos\hat\theta\cdot \hat\theta'(x)=\fr{\cos\hat\theta\varphi''}{1+(\varphi')^2}(x)<0,\ \
\bar\pa^0\sin\omega|_{\widehat{BC}}=\cos\hat\theta(\sin\hat\omega(x))'<0.
\end{align*}

We next analyze the values $\bar\pa^\pm\theta$ on $\widehat{BC}$. Subtracting the last two equations of \eqref{2.9b} from each other and using \eqref{2.7} yield
\begin{align}\label{3.2}
\bar\pa^+\theta-\bar\pa^-\theta+\fr{2\cos^2\omega}{\kappa+\sin^2\omega}\bar\pa^0\sin\omega=0.
\end{align}
Recalling  $\bar\pa^+\theta+\bar\pa^-\theta=2\cos\omega\bar\pa^0\theta$, we have
\begin{align}\label{3.3}
\bar\pa^\pm\theta=\cos\omega\bar\pa^0\theta\mp\fr{\cos^2\omega}{\kappa+\sin^2\omega}\bar\pa^0\sin\omega.
\end{align}
From \eqref{3.3} and \eqref{2.19}, we observe
\begin{align*}
\bar\pa^+\theta|_{\widehat{BC}}=\cos\hat\omega\cos\hat\theta\bigg(\fr{\varphi''}{1+(\varphi')^2} -\fr{\cos\hat\omega(\sin\hat\omega)'}{\kappa+\sin^2\hat\omega}\bigg)(x)<0.
\end{align*}
The fact $\bar\pa^-\theta|_{\widehat{BC}}<0$ follows  from $\bar\pa^0\theta<0$ and $\bar\pa^0\sin\omega<0$ on $\widehat{BC}$, analogously.
\vspace{0.2cm}

For the boundary values on $\widehat{BA}$, we recall the definition of $\bar\pa^+$ in \eqref{2.6} and the relation between $\alpha$ and $(\theta, \omega)$ to acquire $\bar\pa^+=\sin(\tilde\theta+\tilde\omega)\fr{\rm d}{{\rm d}y}$ on $\widehat{BA}$. Hence we find by \eqref{2.9}, \eqref{2.20} and \eqref{2.21} that
\begin{align*}
\bar\pa^+\theta|_{\widehat{BA}}&=\sin(\tilde\theta+\tilde\omega)\tilde\theta'(y) =\fr{\sin(2\tilde\omega)}{\sqrt{1+(\psi')^2}}\bigg(\tilde\Omega'-\fr{(\sin\tilde\omega)'} {2\sin\tilde\omega(\kappa+\sin^2\tilde\omega)}\bigg)(y)<0.
\end{align*}
Moreover, applying \eqref{2.21} again leads to
\begin{align*}
\bar\pa^+\sin\omega|_{\widehat{BA}}=\sin(\tilde\theta+\tilde\omega)(\sin\tilde\omega)'(y) >2H_0\sin\tilde\omega(\kappa+\sin^2\tilde\omega)G(\tilde\omega)(y)\geq0.
\end{align*}
\end{proof}

Using Lemma \ref{lem1}, we have the  the following  estimates on boundary values.
\begin{lem}\label{lem2} With the boundary values in Theorem \ref{thm1}, one has
\begin{align}\label{3.4}
\begin{array}{ll}
 \theta(B)\leq\theta\leq\theta(C),\ \omega(B)\leq\omega\leq\omega(C)=\fr{\pi}{2}, &\mbox{on }\ \widehat{BC}; \ \\
  \theta(A)\leq\theta\leq\theta(B),\ \omega(B)\leq\omega\leq\omega(A)=\fr{\pi}{2}, &\mbox{on }\ \widehat{BA},
\end{array}
\end{align}
where $\theta(C)=\arctan\varphi'(x_1), \theta(B)=\arctan\varphi'(x_2), \theta(A)=\arccot\psi'(y_3)-\fr{\pi}{2}$ and $\omega(B)=[\arccot\psi'(y_2)-\arctan\varphi'(x_2)]\geq\fr{\pi}{4}$.
\end{lem}

Based on the classical existence theory in \cite{LiT, Wang} and the compatibility condition \eqref{2.22}, we have the next local existence theorem around point $B$.
\begin{thm}\label{thm2}
Assume  the boundary values  in Theorem \ref{thm1}. Then the problem \eqref{2.9a} with the boundary data $(H, \theta, \omega )|_{\widehat{CB}}=(\hat{H}, \hat{\theta}, \hat{\omega})(x)$ and $(H, \theta, \omega )|_{\widehat{BA}}=(\tilde{H}, \tilde{\theta}, \tilde{\omega})(y)$ has a smooth solution locally around point $B$.
\end{thm}
\begin{proof}
We first write system \eqref{2.9a} in the following form
\begin{align}\label{3.5}
\mathbf{A}
\left(
\begin{array}{ccc}
H \\
\theta \\
\omega
\end{array}
\right)_x +
\mathbf{B}
\left(
\begin{array}{ccc}
H \\
\theta \\
\omega
\end{array}
\right)_y=\left(
\begin{array}{ccc}
0 \\
\fr{\sin(2\omega)G(\omega)H}{\sin\alpha} \\
\fr{\sin(2\omega)G(\omega)H}{\sin\alpha}
\end{array}
\right),
\end{align}
where the coefficient matrices $\mathbf{A}$ and $\mathbf{B}$ are
\begin{align*}
\mathbf{A}=\left(
\begin{array}{ccc}
\cos\theta &  0  & 0  \\
0 & \cot\alpha & \fr{\cot\alpha\cos^2\omega}{\kappa+\sin^2\omega} \\
0 & \cot\beta & -\fr{\cot\beta\cos^2\omega}{\kappa+\sin^2\omega}
\end{array}
\right),\ \
\mathbf{B}=\left(
\begin{array}{ccc}
\sin\theta &  0  & 0  \\
0 & 1 & \fr{\cos^2\omega}{\kappa+\sin^2\omega} \\
0 & 1 & -\fr{\cos^2\omega}{\kappa+\sin^2\omega}
\end{array}
\right).
\end{align*}
The eigenvalues of \eqref{3.5} are $\Lambda_0=\tan\theta, \Lambda_+=\tan\alpha, \Lambda_-=\tan\beta$.  The corresponding (left) eigenvectors are
\begin{align*}
\ell_0=(1,0,0),\ \ \ell_+=\bigg(0, 1,\fr{\cos^2\omega}{\kappa+\sin^2\omega}\bigg), \ \ \ell_-=\bigg(0, 1,-\fr{\cos^2\omega}{\kappa+\sin^2\omega}\bigg).
\end{align*}
According to the classical theory in \cite{LiT, Wang}, it suffices to check that the following conditions are satisfied:
\begin{align}
(\hat{H}, \hat{\theta}, \hat{\omega})(B)&=(\tilde{H}, \tilde{\theta}, \tilde{\omega})(B),\label{3.6} \\
\bigg(\fr{1}{\Lambda_0-\Lambda_-}\ell_-\cdot\fr{{\rm d}}{{\rm d}x}(\hat{H}, \hat{\theta}, \hat{\omega})\bigg)(B)&= \bigg(\fr{1}{\Lambda_+-\Lambda_-}\ell_-\cdot\fr{\rm d}{{\rm d}y}(\tilde{H}, \tilde{\theta}, \tilde{\omega})\tan\alpha\bigg)(B). \label{3.7}
\end{align}
The equality \eqref{3.6} follows from \eqref{2.17} and \eqref{2.20}.  By a direct calculation, \eqref{3.7} is equivalent to the following equality,
\begin{align*}
\fr{\tan\alpha-\tan\beta}{\tan\theta-\tan\beta}(B)\bigg(\hat\theta' -\fr{\cos^2\hat\omega\hat\omega'}{\kappa+\sin^2\hat\omega}\bigg)(B) =\tan\alpha(B)\bigg(\tilde\theta' -\fr{\cos^2\hat\omega\tilde\omega'}{\kappa+\sin^2\tilde\omega}\bigg)(B),
\end{align*}
or,
\begin{align*}
\fr{2\cos\theta\cos\omega}{\cos\alpha}(B)\bigg(\hat\theta' -\fr{\cos^2\hat\omega\hat\omega'}{\kappa+\sin^2\hat\omega}\bigg)(B) =\tan\alpha(B)\cdot2\tilde\theta'(B).
\end{align*}
Here we have applied the last condition in \eqref{2.20}. Recalling the definition of $\alpha$, we have
$$
\cos\hat{\theta}\cos\hat{\omega}\bigg(\hat\theta' -\fr{\cos^2\hat\omega\hat\omega'}{\kappa+\sin^2\hat\omega}\bigg)(B) =\fr{1}{\sqrt{1+(\psi')^2}}\tilde\theta'(B),
$$
which is true thanks to \eqref{2.22}.  The proof of Theorem \ref{thm2} is completed.
\end{proof}

\subsection{A priori estimates on the global solution}

In this subsection, we derive a priori estimates of the solutions, which serve to extend the local solution in Theorem \ref{thm2} to the global domain $ABC$. All estimates are based on the assumptions of Theorem \ref{thm1}. From now on, we just consider system \eqref{2.9b}.
For convenience of presentation, we denote $D_\eps=\{(x,y)|\ \cos\omega(x,y)>\eps\}\cap ABC$.

\begin{lem}\label{lem4}
Suppose that $(\theta, \omega)\in C^2(D_\eps)$ is a solution to system \eqref{2.9b} with the boundary data $(\theta, \omega )|_{\widehat{CB}}=(\hat{\theta}, \hat{\omega})(x)$ and $(\theta, \omega )|_{\widehat{BA}}=(\tilde{\theta}, \tilde{\omega})(y)$. Then one has, for all $(x,y)\in D_\eps$,
\begin{align}\label{3.8}
\begin{array}{l}
\bar\pa^+\theta<0,\quad \bar\pa^-\theta<0, \\
\bar\pa^+\omega>0,\quad \bar\pa^-\omega<0,\quad \bar\pa^+\Xi>0, \quad \bar\pa^-\Xi<0.
\end{array}
\end{align}
\end{lem}
\begin{proof}
We use the contradiction argument to prove this lemma. Thanks to Lemma \ref{lem1}, we know that  $\bar\pa^+\theta<0$ in the region near $B$ in $D_\eps$. Suppose that there is a point $P$ in $D_\eps$ that is the first time such that $\bar\pa^+\theta=0$, i.e., $\bar\pa^+\theta<0$ in the region $D_P:=D_\eps\cap\{(x,y)|\cos\omega(x,y)>\cos\omega(P)\}$. The proof consists of two cases.
\vspace{0.2cm}

\noindent \textbf{Case 1.} For any point $(x,y)\in D_p$, there holds $\bar\pa^-\theta<0$. In this case, from the point $P$, we draw a $\Lambda_-$-characteristic curve, called $\Gamma^-$, up to the boundary $\widehat{ABC}$ at a point $P_1$. Recalling the second equation in \eqref{2.9b}, we have
\begin{align*}
\fr{\cos^2\omega}{\kappa+\sin^2\omega}\bar\pa^-\omega=\bar\pa^-\theta-H_0\sin(2\omega)G(\omega),
\end{align*}
where means that $\bar\pa^-\omega<0$ on $\Gamma^-\setminus\{P\}$. Hence the direction of $\bar\pa^-$ is from $P$ to $P_1$ along $\Gamma^-$. It follows that $\bar\pa^-\bar\pa^+\theta<0$ at $P$. However, from the first equation of \eqref{2.16} one has
$$
\bar\pa^-\bar\pa^+\theta|_{P}=\bigg(\fr{H_0(\kappa+\sin^2\omega)G(\omega)} {\cos^2\omega}\cos(2\omega)\bar\pa^-\theta\bigg)(P)\geq0,
$$
which leads to a contradiction.
\vspace{0.2cm}

\noindent \textbf{Case 2.} There exists a point $Q\in D_p$ such that $\bar\pa^-\theta=0$. Then
from the point $Q$, we draw a $\Lambda_+$- characteristic curve, called $\Gamma^+$, up to the boundary $\widehat{BC}$ at a point $Q_1$. Due to $\bar\pa^-\theta<0$ at $Q_1$ by Lemma \ref{lem1}, without the loss of
generality, we assume $\bar\pa^-\theta<0$ on $\Gamma^+\setminus\{Q\}$. Otherwise we can take a point $Q_2$ between $Q$ and $Q_1$ on $\Gamma^+$ and then use $Q_2$ instead of $Q$ in the analysis below. According to the above assumption, we see that $\bar\pa^+\bar\pa^-\theta>0$ at $Q$. On the other hand, we recall the second equation of \eqref{2.16} to find
$$
\bar\pa^+\bar\pa^-\theta|_{Q}=\bigg(-\fr{H_0(\kappa+\sin^2\omega)G(\omega)} {\cos^2\omega}\cos(2\omega)\bar\pa^+\theta\bigg)(Q)\leq0.
$$
This results in  a contradiction. Therefore, we have $\bar\pa^+\theta<0, \bar\pa^-\theta<0$.

The conclusions $\bar\pa^+\omega>0, \bar\pa^-\omega<0$, $\bar\pa^+\Xi>0$ and $\bar\pa^-\Xi<0$ on the region $\overline{D_P}$ follow directly from the equations \eqref{2.9b} and \eqref{2.14}.  Thus we complete the proof this lemma.
\end{proof}

With the aid of Lemma \ref{lem4}, we can get the $C^0$ estimates on $(\theta,\omega)$.
\begin{lem}\label{lem5}
Suppose that the assumptions in Theorem \ref{thm1} hold and $(\theta, \omega)\in C^2(D_\eps)$ is a solution to system \eqref{2.9b} with the boundary data $(\theta, \omega )|_{\widehat{CB}}=(\hat{\theta}, \hat{\omega})(x)$ and $(\theta, \omega )|_{\widehat{BA}}=(\tilde{\theta}, \tilde{\omega})(y)$. Then one has
\begin{align}\label{3.9}
\theta(A)\leq\theta(x,y)\leq\theta(C),\quad \fr{\pi}{4}\leq\omega(B)\leq\omega(x,y)<\fr{\pi}{2},\quad \forall\ (x,y)\in D_\eps.
\end{align}
\end{lem}

Moreover, we have the properties about $\Lambda_\pm$-characteristics.
\begin{lem}\label{lem6}
Suppose that $(\theta, \omega)\in C^2(D_\eps)$ is a solution to system \eqref{2.9b} with the boundary data $(\theta, \omega )|_{\widehat{CB}}=(\hat{\theta}, \hat{\omega})(x)$ and $(\theta, \omega )|_{\widehat{BA}}=(\tilde{\theta}, \tilde{\omega})(y)$. Denote the $\Lambda_\mp$-characteristics in $D_\eps$ by $\Gamma^-: y=y(x)$ and  $\Gamma^+: x=x(y)$, respectively. Then the curves $\Gamma^-$ are convex and the curves $\Gamma^+$ are concave.
\end{lem}
\begin{proof}
Owing to the relations between $(\theta, \omega)$ and $(\alpha, \beta)$ \eqref{2.4}, we obtain by \eqref{3.8}
\begin{align}\label{3.10}
\bar\pa^-\alpha=\bar\pa^-\theta+\bar\pa^-\omega<0,\quad \bar\pa^+\beta=\bar\pa^+\theta-\bar\pa^+\omega<0,\quad \forall\ (x,y)\in D_\eps.
\end{align}
Moreover, we use \eqref{2.4} again to write system \eqref{2.9b} as
\begin{align*}
\left\{
   \begin{array}{l}
    (\kappa+1)\bar{\partial}^+\alpha+(\kappa-\cos(2\omega))\bar{\partial}^+\beta=
2H_0(\kappa+\sin^2\omega)\sin(2\omega)G(\omega), \\
    (\kappa-\cos(2\omega))\bar{\partial}^-\alpha+(\kappa+1)\bar{\partial}^-\beta=
2H_0(\kappa+\sin^2\omega)\sin(2\omega)G(\omega),
  \end{array}
\right.
\end{align*}
which, together with \eqref{3.10},  yield
\begin{align*}
\bar{\pa}^+\alpha&=\fr{(\cos(2\omega)-\kappa)}{\kappa+1}\bar{\pa}^+\beta +\fr{2H_0(\kappa+\sin^2\omega)\sin(2\omega)G(\omega)}{\kappa+1}>0, \\
\bar{\pa}^-\beta&=\fr{(\cos(2\omega)-\kappa)}{\kappa+1}\bar{\pa}^-\alpha +\fr{2H_0(\kappa+\sin^2\omega)\sin(2\omega)G(\omega)}{\kappa+1}>0.
\end{align*}
Therefore  with relations $\bar\pa^+=\sin\alpha\fr{{\rm d}}{{\rm d}y}$ and $\bar\pa^-=\cos\beta\fr{{\rm d}}{{\rm d}x}$, we conclude  the lemma.
\end{proof}

We proceed to derive the $C^1$ estimates on $(\theta,\omega)$. Let
\begin{align}\label{3.11}
M_0=\max\{\max_{\widehat{BA},\widehat{BC}}\bar{\partial}^+\Xi,\
 -\min_{\widehat{BC}}\bar{\partial}^-\Xi\}.
\end{align}
Then we have:

\begin{lem}\label{lem7}
Assume that $(\theta, \omega)\in C^2(D_\eps)$ is a solution to system \eqref{2.9b} with the boundary data $(\theta, \omega )|_{\widehat{CB}}=(\hat{\theta}, \hat{\omega})(x)$ and $(\theta, \omega )|_{\widehat{BA}}=(\tilde{\theta}, \tilde{\omega})(y)$. Then one has
\begin{align}\label{3.12}
0<\bar{\pa}^+\Xi<M_0,\quad -M_0<\bar{\pa}^-\Xi<0,\quad \forall\ (x,y)\in D_\eps.
\end{align}
\end{lem}
\begin{proof}
In light of  Lemma \ref{lem4}, it suffices to prove $\bar{\pa}^+\Xi<M_0$ and $\bar{\pa}^-\Xi>-M_0$. Suppose the curve $\ell_1:=\{(x,y)|\cos\omega(x,y)=\eps_1>\eps\}$ is the first time that either of $\bar{\pa}^+\Xi$ and $-\bar{\pa}^-\Xi$
touches the bound $M_0$ for the solution in $D_\eps$. Without the loss of generality,
we assume that $\bar{\pa}^+\Xi=M_0$ at the point $P$ on $\ell_1$. Then we put it into the first equation of \eqref{2.15} to obtain
\begin{align*}
\bar{\pa}^-\bar{\pa}^+\Xi|_P >\fr{\kappa M_0+H_0(\kappa+\sin^2\omega)G(\omega)}{\cos^2\omega}[M_0 +\bar{\pa}^-\Xi]+\fr{M_0}{\cos^2\omega}[M_0+\bar{\pa}^-\Xi]\geq0.
\end{align*}
On the other hand, we note the direction of $\bar\pa^-$ to know that $\bar{\pa}^-\bar{\pa}^+\Xi|_P\leq0$, which leads to a contradiction.
The proof of the lemma is complete.
\end{proof}
For later use, we show that $\bar{\pa}^+\Xi$ and $-\bar{\pa}^-\Xi$ have a positive lower bound independent of $\eps$ in the region $D_\eps$. For any point $(x,y)\in D_\eps$, we draw $\Lambda_\mp$-characteristic curves up to the boundaries $\widehat{ABC}$ and $\widehat{BC}$ at points $B_1$ and $B_2$, respectively, see Fig. \ref{fig3}. We have the following lemma.

\begin{lem}\label{lem8}
Assume that $(\theta, \omega)\in C^2(D_\eps)$ is a solution to system \eqref{2.9b} with the boundary data $(\theta, \omega )|_{\widehat{CB}}=(\hat{\theta}, \hat{\omega})(x)$ and $(\theta, \omega )|_{\widehat{BA}}=(\tilde{\theta}, \tilde{\omega})(y)$. Then  $\bar{\pa}^+\Xi$ and $-\bar{\pa}^-\Xi$ satisfy
\begin{align}\label{3.13}
0<m_0e^{-2d\overline{M}}\leq\bar{\pa}^+\Xi<M_0,\quad 0<m_0e^{-2d\overline{M}}\leq-\bar{\pa}^-\Xi<M_0,\quad \forall\ (x,y)\in D_\eps,
\end{align}
where $d$ is the diameter of the domain $ABC$ and
\begin{align*}
m_0=\min\{\min_{\widehat{BB_1}, \widehat{BB_2}}\bar\pa^+\Xi,\  \min_{\widehat{BB_2}}(-\bar\pa^-\Xi)\}, \ \ \overline{M}=2(\kappa+2)M_0+2(\kappa+1)H_0.
\end{align*}
\end{lem}
\begin{proof}
To prove \eqref{3.13}, we first rewrite the characteristic decompositions \eqref{2.15} as
\begin{align}\label{3.14}
\left\{
\begin{array}{l}
\bar{\pa}^-\bar{\pa}^+\Xi =\fr{-\cos(2\omega)[\kappa\bar{\pa}^+\Xi+H_0(\kappa+\sin^2\omega)G(\omega)] +\cos^2(2\omega)\bar{\pa}^+\Xi}{\cos^2\omega}(\bar{\pa}^+\Xi+\bar{\pa}^-\Xi) \\ \qquad \qquad \qquad +[2(\kappa+2\sin^2\omega)\bar{\pa}^+\Xi+2H_0(\kappa+\sin^2\omega)G(\omega)]\bar\pa^+\Xi, \\[5pt]
-\bar{\pa}^+(-\bar{\pa}^-\Xi) =-\fr{-\cos(2\omega)[\kappa(-\bar{\pa}^-\Xi)+H_0(\kappa+\sin^2\omega)G(\omega)] +\cos^2(2\omega)(-\bar{\pa}^-\Xi)}{\cos^2\omega}(\bar{\pa}^+\Xi+\bar{\pa}^-\Xi) \\ \qquad \qquad \qquad +[2(\kappa+2\sin^2\omega)(-\bar{\pa}^-\Xi)+2H_0(\kappa+\sin^2\omega)G(\omega)](-\bar\pa^-\Xi).
\end{array}
\right.
\end{align}
Here we used the directions $\bar\pa^-$ and $-\bar\pa^+$ since they  both point to the boundary $\widehat{ABC}$.
According to Lemmas \ref{lem5} and \ref{lem7}, we see that the terms are nonnegative
$$
\fr{-\cos(2\omega)[\kappa(\pm\bar{\pa}^\pm\Xi)+H_0(\kappa+\sin^2\omega)G(\omega)] +\cos^2(2\omega)(\pm\bar{\pa}^\pm\Xi)}{\cos^2\omega} \geq 0.
$$
 Let $r(x,y)$ and $s(x,y)$ be two smooth positive functions satisfying
$$
\bar\pa^-r=1,\quad -\bar\pa^+s=1.
$$
Denote
$$
R=e^{-\overline{M}r}\bar\pa^+\Xi, \quad S=e^{-\overline{M}s}(-\bar\pa^-\Xi).
$$
\begin{figure}[htbp]
\begin{center}
\includegraphics[scale=0.6]{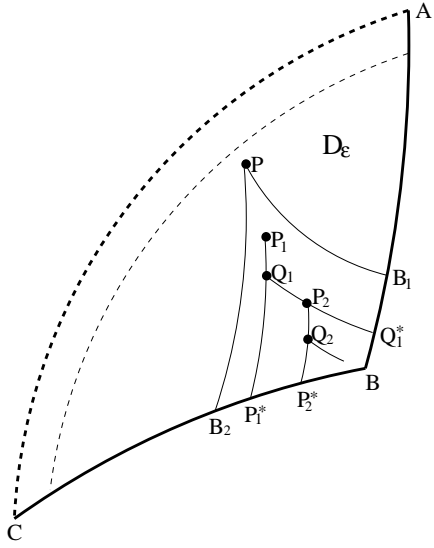}
\caption{\footnotesize The region $PB_1BB_2$.}\label{fig3}
\end{center}
\end{figure}
Then it follows from \eqref{3.14}
\begin{align}\label{3.15}
\left\{
\begin{array}{l}
\bar{\pa}^-R =\fr{-\cos(2\omega)[\kappa\bar{\pa}^+\Xi+H_0(\kappa+\sin^2\omega)G(\omega)] +\cos^2(2\omega)\bar{\pa}^+\Xi}{\cos^2\omega}e^{-\overline{M}r}(\bar{\pa}^+\Xi+\bar{\pa}^-\Xi) \\ \qquad \qquad \qquad +[2(\kappa+2\sin^2\omega)\bar{\pa}^+\Xi+2H_0(\kappa+\sin^2\omega)G(\omega)-\overline{M}]R=:F_1, \\[5pt]
-\bar{\pa}^+S =-\fr{-\cos(2\omega)[\kappa(-\bar{\pa}^-\Xi)+H_0(\kappa+\sin^2\omega)G(\omega)] +\cos^2(2\omega)(-\bar{\pa}^-\Xi)}{\cos^2\omega}e^{-\overline{M}s}(\bar{\pa}^+\Xi+\bar{\pa}^-\Xi) \\ \qquad \qquad \qquad +[2(\kappa+2\sin^2\omega)(-\bar{\pa}^-\Xi)+2H_0(\kappa+\sin^2\omega)G(\omega)-\overline{M}]S=:F_2.
\end{array}
\right.
\end{align}
We divide the proof into two cases.

\noindent \textbf{Case 1.} For any point $P\in D_\eps$, if $\bar{\pa}^+\Xi+\bar{\pa}^-\Xi\leq0$, i.e., $-\bar{\pa}^-\Xi\geq\bar{\pa}^+\Xi$, entirely in the region $PB_1BB_2$, then one has by the first equation of \eqref{3.15} and the definition of $\overline{M}$
\begin{align*}
\bar{\pa}^-R=F_1\leq0,
\end{align*}
from which we have $R|_P\geq R_{B_1}$. That is, there hold
\begin{align*}
-\bar{\pa}^-\Xi|_P\geq\bar{\pa}^+\Xi|_P\geq \bar{\pa}^+\Xi|_{B_1}e^{-\overline{M}[r(B_1)-r(P)]}\geq m_0e^{-2d\overline{M}}.
\end{align*}

\noindent \textbf{Case 2.} If there exists a point, say $P_1$, in the region $PB_1BB_2$ such that $\bar{\pa}^+\Xi+\bar{\pa}^-\Xi>0$ at $P_1$, then from the point $P_1$, we draw a $\Lambda_+$-characteristic curve, called $\Gamma^{+}_1$, up to the boundary $\widehat{BB_2}$ at a point $P^{*}_1$, see Fig. \ref{fig3}. It is obvious that $F_2<0$ on $\Gamma^{+}_1$ near the point $P_1$.
If $F_2\leq0$ always holds on $\Gamma^{+}_1$, then we apply the second equation of \eqref{3.15} to get
\begin{align*}
-\bar{\pa}^+S=F_2 \leq0
\end{align*}
on $\Gamma^{+}_1$. Noting the direction $-\bar\pa^+$ gives $S|_{P_1}\geq S|_{P^{*}_1}$. Thus we have
$$
\bar{\pa}^+\Xi|_{P_1}>(-\bar{\pa}^-\Xi)|_{P_1} \geq(-\bar{\pa}^-\Xi)|_{P^{*}_1}e^{-\overline{M}[s(P^{*}_1)-s(P_1)]}\geq m_0e^{-2d\overline{M}}.
$$
If there exist some points such that $F_2>0$ on $\Gamma^{+}_1$, then we take a
neighborhood $\mathbb{N}_1$ of $P_1$ so that $F_2<0$ on $\mathbb{N}_1\cap\Gamma^{+}_1$ and $F_2=0$ at $Q_1:=\partial(\mathbb{N}_1\cap\Gamma^{+}_1)$. Thus one obtains
\begin{align}\label{3.16}
\bar{\pa}^+\Xi|_{P_1}>(-\bar{\pa}^-\Xi)|_{P_1} \geq(-\bar{\pa}^-\Xi)|_{Q_1}e^{-\overline{M}[s(Q_1)-s(P_1)]}.
\end{align}
Due to the fact $F_2=0$ at $Q_1$, we see that $\bar{\pa}^+\Xi+\bar{\pa}^-\Xi<0$ at $Q_1$. Hence we get from \eqref{3.16}
\begin{align}\label{3.17}
\bar{\pa}^+\Xi|_{P_1}>(-\bar{\pa}^-\Xi)|_{P_1} \geq \bar{\pa}^+\Xi|_{Q_1}e^{-\overline{M}[s(Q_1)-s(P_1)]}.
\end{align}
We now draw a $\Lambda_-$-characteristic curve from the point $Q_1$, called $\Gamma^{-}_1$, up to the boundary $\widehat{ABC}$ at a point $Q^{*}_1$. Noting $F_1<0$ at $Q_1$, if $F_1\leq0$ always holds on $\Gamma^{-}_1$, then we repeat the process for Case 1 to conclude that
$$
\bar{\pa}^+\Xi|_{Q_1}\geq\bar{\pa}^+\Xi|_{Q^{*}_1}e^{-\overline{M}[r(Q^{*}_1)-r(Q_1)]}\geq m_0e^{-\overline{M}[r(Q^{*}_1)-r(Q_1)]}.
$$
Inserting the above into \eqref{3.17} gives
$$
\bar{\pa}^+\Xi|_{P_1}>(-\bar{\pa}^-\Xi)|_{P_1} \geq m_0e^{-\overline{M}\{[s(Q_1)-s(P_1)]+[r(Q^{*}_1)-r(Q_1)]\}}\geq m_0e^{-2d\overline{M}}.
$$
Otherwise, there exists a neighborhood $\mathbb{N}_2$ of $Q_1$ on $\Gamma^{-}_1$ such that $F_1<0$ holds in $\mathbb{N}_2\cap\Gamma^{-}_1$ and $F_1=0$ at $P_2:=\partial(\mathbb{N}_2\cap\Gamma^{-}_1)$. Thus one has $\bar\pa^-R\leq0$ on $\mathbb{N}_2\cap\Gamma^{-}_1$ which leads to
$$
\bar{\pa}^+\Xi|_{Q_1}\geq\bar{\pa}^+\Xi|_{P_2}e^{-\overline{M}[r(P_2)-r(Q_1)]}.
$$
We put it into \eqref{3.17} to find that
\begin{align*}
\bar{\pa}^+\Xi|_{P_1}>(-\bar{\pa}^-\Xi)|_{P_1} &\geq \bar{\pa}^+\Xi|_{P_2}e^{-\overline{M}\{[s(Q_1)-s(P_1)]+[r(P_2)-r(Q_1)]\}}   \\
&> (-\bar{\pa}^-\Xi)|_{P_2}e^{-\overline{M}\{[s(Q_1)-s(P_1)]+[r(P_2)-r(Q_1)]\}}.
\end{align*}
Here we used the inequality $(\bar{\pa}^+\Xi+\bar{\pa}^-\Xi)>0$ at $P_2$ which follows from the expression of $F_1$ and the fact $F_1=0$ at $P_2$. To estimate $(-\bar{\pa}^-\Xi)|_{P_2}$, we draw a positive characteristic curve from the point $P_2$, called $\Gamma^{+}_2$, up to the boundary $\widehat{BB_2}$ at a point $P^{*}_2$. We note $F_2<0$ at $P_2$ and then
repeat the above process in Case 2 to complete the proof the lemma.
\end{proof}

In light of  Lemma \ref{lem7}, we can establish the gradient estimates of solutions.
\begin{lem}\label{lem9}
Suppose that  $(\theta, \omega)\in C^2(D_\eps)$ is a solution to system \eqref{2.9b} with the boundary data $(\theta, \omega )|_{\widehat{CB}}=(\hat{\theta}, \hat{\omega})(x)$ and $(\theta, \omega )|_{\widehat{BA}}=(\tilde{\theta}, \tilde{\omega})(y)$. Then one has
\begin{align}\label{3.18}
\|(\theta, \omega)\|_{C^1(D_\eps)}\leq\fr{M}{\eps^2},
\end{align}
where $M$ is a positive constant independent of $\eps$.
\end{lem}
\begin{proof}
We first recall \eqref{2.7} and \eqref{2.14} to find that
\begin{align}\label{3.18a}
\theta_x=\sin\beta\bar\pa^+\Xi+\sin\alpha\bar\pa^-\Xi,\quad \theta_y=-\cos\beta\bar\pa^+\Xi-\cos\alpha\bar\pa^-\Xi,
\end{align}
which along with \eqref{3.12} arrives at
\begin{align}\label{3.19}
|\theta_x|+|\theta_y|\leq 2M_0.
\end{align}
Furthermore, we recall \eqref{2.14a} to acquire
\begin{align*}
|\bar{\pa}^\pm\omega|\leq\fr{2(\kappa+1)(M_0+H_0)}{\eps}.
\end{align*}
which together with \eqref{2.7} yields
\begin{align}\label{3.21}
|\omega_x|+|\omega_y|\leq\fr{2(|\bar{\pa}^+\omega|+|\bar{\pa}^-\omega|)}{\sin(2\omega)} \leq\fr{8(\kappa+1)(M_0+H_0)}{\eps^2}.
\end{align}
Combining with \eqref{3.19} and \eqref{3.21} finishes the proof of the lemma.
\end{proof}
In addition, we have  $C^{1,1}$ estimates on the solutions.
\begin{lem}\label{lem10}
Suppose that  $(\theta, \omega)\in C^2(D_\eps)$ is a solution to system \eqref{2.9b} with the boundary data $(\theta, \omega )|_{\widehat{CB}}=(\hat{\theta}, \hat{\omega})(x)$ and $(\theta, \omega )|_{\widehat{BA}}=(\tilde{\theta}, \tilde{\omega})(y)$. Then one has
\begin{align}\label{3.22}
\|(\theta, \omega)\|_{C^{1,1}(D_\eps)}\leq\fr{M}{\eps^5},
\end{align}
where $M$ is a positive constant, independent of $\eps$.
\end{lem}
\begin{proof}
This lemma follows from the characteristic decompositions \eqref{2.15} and \eqref{2.16}. We first use the characteristic decompositions \eqref{2.15} and \eqref{2.16} and the gradient estimates \eqref{3.18} to obtain
\begin{align}\label{3.23}
|\bar{\pa}^+\bar{\pa}^-\theta|+|\bar{\pa}^-\bar{\pa}^+\theta|\leq\fr{M}{\eps},\quad |\bar{\pa}^+\bar{\pa}^-\Xi|+|\bar{\pa}^-\bar{\pa}^+\Xi|\leq\fr{M}{\eps^2}
\end{align}
for some positive constant $M$ independent of $\eps$. We next estimate the terms $|\bar{\pa}^\pm\bar{\pa}^\pm\theta|$ and $|\bar{\pa}^\pm\bar{\pa}^\pm\Xi|$. For the term $|\bar{\pa}^+\bar{\pa}^+\theta|$, we differentiate the first equation of \eqref{2.16} and employ
the commutator relation between $\bar{\pa}^-$ and $\bar{\pa}^+$
\begin{align*}
\bar{\pa}^-\bar{\pa}^+-\bar{\pa}^+\bar{\pa}^- =\fr{\cos(2\omega)\bar{\pa}^+\beta-\bar{\pa}^-\alpha}{\sin(2\omega)}\bar{\pa}^- -\fr{\bar{\pa}^+\beta-\cos(2\omega)\bar{\pa}^-\alpha}{\sin(2\omega)}\bar{\pa}^+,
\end{align*}
to achieve by performing a direct calculation
\begin{align}\label{3.24}
\bar{\pa}^-(\bar{\pa}^+\bar{\pa}^+\theta)+T_1\bar{\pa}^+\bar{\pa}^+\theta=T_2,
\end{align}
where $T_1$ and $T_2$ are lower-order terms which can be estimated
$$
|T_1|\leq\fr{M}{\eps^2},\quad |T_2|\leq \fr{M}{\eps^3}
$$
for some positive constant $M$ independent of $\eps$. Integrating \eqref{3.24} leads to
\begin{align}\label{3.25}
|\bar{\pa}^+\bar{\pa}^+\theta|\leq\fr{M}{\eps^3}.
\end{align}
With similar arguments for $|\bar{\pa}^-\bar{\pa}^-\theta|$ and $|\bar{\pa}^\pm\bar{\pa}^\pm\Xi|$, one proceeds to obtain
\begin{align}\label{3.26}
|\bar{\pa}^-\bar{\pa}^-\theta|\leq\fr{M}{\eps^3},\quad |\bar{\pa}^\pm\bar{\pa}^\pm\Xi|\leq\fr{M}{\eps^4}.
\end{align}
We combine \eqref{3.23} and \eqref{3.25}-\eqref{3.26} and apply \eqref{2.7} to obtain \eqref{3.22}. The proof is completed.
\end{proof}

\subsection{Global existence of solutions}

With the help of a priori estimates established above, we extend the local solution in Theorem \ref{thm2} to the whole region $ABC$. Again we assume the conditions in Theorem \ref{thm1} hold.
We first have the lemma directly following from the facts $\pm\bar\pa^\pm\omega>0$  in view of Lemma \ref{lem4}.
\begin{lem}\label{lem11}
Assume that $(\theta, \omega)\in C^2(D_\eps)$ is a solution to system \eqref{2.9b} with the boundary data $(\theta, \omega )|_{\widehat{CB}}=(\hat{\theta}, \hat{\omega})(x)$ and $(\theta, \omega )|_{\widehat{BA}}=(\tilde{\theta}, \tilde{\omega})(y)$. Then the level curves of $\omega$ in $D_\eps$ are $C^1$ and non-characteristic.
\end{lem}

Note that the level sets of $\omega$ are non-characteristic. Then we can take the level curves of $\omega$ as the
``Cauchy supports''  and  extend the local solution to the whole region $ABC$. For this purpose, we  introduce the definition of strong determinate domain.
\begin{defn}\label{def1}
Let $D'$ be a closed domain bounded by $\widehat{BA}$, $\widehat{BC}$ and $\ell$, where the curve $\ell$ is a level set of $\omega$ intersecting with $\widehat{BA}$, $\widehat{BC}$ and stays in the domain $ABC$. Suppose that $\omega\in C^1(D')$ and $\omega\in[\fr{\pi}{4},\fr{\pi}{2})$.
We call $D'$ is a strong determinate domain of $\omega$ if and only if, for all $(x_0,y_0)\in D'$, the two characteristic curves, defined by
$$
\begin{array}{ll}
\fr{{\rm d}y}{{\rm d}x}=\tan\beta(x,y), \ \ & \mbox{if }  x\geq x_0,\\
\fr{{\rm d}x}{{\rm d}y}=\tan\alpha(x,y), & \mbox{ if } y\leq y_0,
\end{array}
$$
intersect only with $\widehat{BA}$ and $\widehat{BC}$.
\end{defn}

The following lemma is important and it immediately implies  the global existence theorem.
\begin{lem}\label{lem12}
 For any $\bar{x}\in(x_C,x_B]$ on the curve $\widehat{BC}:\ y=\varphi(x)$, there exists a smooth curve $\ell_{\bar{x}}:\ \omega(x,y)=\omega(\bar{x},\varphi(\bar{x}))$ is $C^1$ inside the region $ABC$ to encircle a domain $D_{\bar\omega}$ with $\widehat{BA}$ and $\widehat{BC}$ such that

\noindent {\rm (i)} $D_{\bar\omega}$ is a strong determinate domain of $\omega$;

\noindent {\rm (ii)} the boundary value problem \eqref{2.9b}, subject to  the boundary data $(\theta, \omega )|_{\widehat{BC}}=(\hat{\theta}, \hat{\omega})(x)$ and $(\theta, \omega )|_{\widehat{BA}}=(\tilde{\theta}, \tilde{\omega})(y)$, has a supersonic solution $(\theta,\omega)\in C^1(D_{\bar\omega})$.
\end{lem}
\begin{proof}
We employ the technique proposed by Dai and Zhang \cite{Dai} to show this lemma. Let $\mathcal{S}$ be
the set consisting of all elements of $(x_C,x_B]$ which satisfies the above assertions listed in this lemma. Thanks to Lemma \ref{lem11}, we see that, if $\tilde{x}\in \mathcal{S}$, then $[\tilde{x},x_B]\subset \mathcal{S}$. Thus it suffices to prove that the set $\mathcal{S}$ is not empty and $\inf \mathcal{S}=x_C$. The non-emptiness of set $\mathcal{S}$ follows obviously from the local existence theorem \ref{thm2}.

We are going to  verify $\inf \mathcal{S}=x_C$ by applying the contradiction argument.
Suppose that $\inf \mathcal{S}=\hat{x}>x_C$. The verification is divided into two steps. We will show that $\hat{x}\in \mathcal{S}$ in Step 1 and that there exists a small $\eps>0$ such that $[\hat{x}-\eps,\hat{x}]\subset \mathcal{S}$ in Step 2.
\vspace{0.2cm}

\noindent \textbf{Step 1.} In view of  the definition for $\hat{x}$, there exists a monotone decreasing sequence $\{x_i\}^{\infty}_{i=1}\subset \mathcal{S}$ such that $\displaystyle\lim_{i\rightarrow\infty}x_i=\hat{x}$. Then for every $x_i$, there exists a $C^1$ smooth curve $\ell_i: \omega(x,y)=\omega(x_i, \varphi(x_i))$ inside the domain $ABC$ to encircle $\Omega_i$ with $\widehat{BA}$ and $\widehat{BC}$ such that

\noindent {\rm (i)} $\Omega_i$ is a strong determinate domain of $\omega$;

\noindent {\rm (ii)} the boundary value problem \eqref{2.9b} with the boundary data $(\theta, \omega )|_{\widehat{CB}}=(\hat{\theta}, \hat{\omega})(x)$ and $(\theta, \omega )|_{\widehat{BA}}=(\tilde{\theta}, \tilde{\omega})(y)$ has a supersonic solution $(\theta,\omega)\in C^1(\Omega_i)$.

In view of the uniqueness of solutions and Lemma \ref{lem11}, we observe that the curve $\ell_i$ is below $\ell_j$ ($i<j$) along the directions $\bar{\partial}^+$ and $-\bar{\partial}^-$,  implying that the sequence $\{\ell_{i}\}^{\infty}_{i=1}$ is a monotone increasing along the directions $\bar{\partial}^+$ and $-\bar{\partial}^-$. Therefore, there exists a curve $\hat{\ell}: \omega(x, y)=\omega(\hat{x}, \varphi(\hat{x}))$ defined by the limit of $\ell_i$ along the characteristic directions. Denote by $\hat{\Omega}$ the closed domain bounded by $\widehat{BA}$, $\widehat{BC}$ and $\hat{\ell}$. Then the boundary value problem \eqref{2.9b}, subject to  the boundary data $(\theta, \omega )|_{\widehat{BC}}=(\hat{\theta}, \hat{\omega})(x)$ and $(\theta, \omega )|_{\widehat{BA}}=(\tilde{\theta}, \tilde{\omega})(y)$ has a $C^1$ solution in $\hat{\Omega}\backslash\hat{\ell}$, and $\hat{\Omega}\backslash\hat{\ell}$ is a strong determinate domain of $\omega$.

Due to Lemma \ref{lem10}, there exists a constant $M$ independent of $i$ such that for $i:1\leq i<\infty$
\begin{align}
\parallel(\theta,\omega)\parallel_{C^{1,1}(\Omega_i)}\leq M,
\nonumber
\end{align}
and
\begin{align}
\parallel(\theta,\omega)\parallel_{C^{1,1}
(\hat{\Omega}\backslash\hat{\ell})}\leq
M.
\nonumber
\end{align}
Combining with the above and Lemmas \ref{lem10} and \ref{lem11}, we know that the sequence $\ell_i$ are $C^{1,1}$ and the bounds of $\omega_x$ and $\omega_y$ are independent of $i$. Then by employing the Arzela-Ascoli theorem,
one finds that the curve $\hat{\ell}$ is $C^1$. We denote $(\theta_i,\omega_i)=(\theta,\omega)|_{\ell_i}$ and then see that
$$
\parallel(\theta_i,\omega_i)\parallel_{C^{1,1}(\Omega_i)}\leq M,
$$
where $C$ is a constant independent of $i$. It follows by the Arzela-Ascoli theorem that there exists
$(\hat{\theta},\hat{\omega})\in C^1$ such that
$$
\lim_{i\rightarrow\infty}(\theta_i,\omega_i)=(\hat{\theta},\hat{\omega}).
$$
We  let $(\theta,\omega)|_{\hat{\ell}}=(\hat{\theta},\hat{\omega})$ and extend $(\theta,\omega)(x,y)$ to the domain $\hat{\Omega}$. Hence the boundary value problem \eqref{2.9b}, subject to  the boundary data $(\theta, \omega )|_{\widehat{BC}}=(\hat{\theta}, \hat{\omega})(x)$ and $(\theta, \omega )|_{\widehat{BA}}=(\tilde{\theta}, \tilde{\omega})(y)$, has a $C^1$ solution in $\hat{\Omega}$. It is easy to see by the fact $\pm\bar{\partial}^{\pm}\omega>0$ that $\hat{x}\in \mathcal{S}$.

\noindent \textbf{Step 2.} Due to $\hat{x}\in \mathcal{S}$, there is a $C^1$ solution in a closed domain near
$\hat{\ell}$, denoted by $\mathcal{E_{\hat{\ell}}}$. Moreover, for any point $(\tilde{x},\tilde{y})\in \hat{\ell}$, we find that the curve $\hat{\ell}$ can be written as $y=\hat{y}(x)$ in a small neighborhood of $(\tilde{x},\tilde{y})$. In this neighborhood, let us denote the limits of $(\theta_{x},\omega_x,\theta_y,\omega_y)$ on the upper and the lower sides of the curve $\hat{\ell}$ by $(\theta_{x}^u,\omega_{x}^u,\theta_{y}^u,\omega_{y}^u)$ and
$(\theta_{x}^\ell,\omega_{x}^\ell,\theta_{y}^\ell,\omega_{y}^\ell)$, respectively.
Then, both of these vector-valued functions are solutions to the system
\begin{align}\label{3.27}
M(x)U(x)=N(x),
\end{align}
where $U(x)=(u_1,u_2,u_3,u_4)^\top(x)$ is a unknown function,
\begin{align}
M(x)=\left(
 \begin{array}{cccc}
   1 & \fr{\cos^2\omega}{\kappa+\sin^2\omega} & \tan\alpha & \fr{\cos^2\omega\tan\alpha}{\kappa+\sin^2\omega} \\
   1 & -\fr{\cos^2\omega}{\kappa+\sin^2\omega} & \tan\beta & -\fr{\cos^2\omega\tan\beta}{\kappa+\sin^2\omega} \\
   1 & 0 & \hat{y}'(x) & 0 \\
   0 & 1 & 0 & \hat{y}'(x)
 \end{array}
 \right),\quad
N(x)=\left(
  \begin{array}{c}
  \fr{H_0\sin(2\omega)G(\omega)}{\cos\alpha}\\
  \fr{H_0\sin(2\omega)G(\omega)}{\cos\beta}\\
  \hat{\theta}'(x)\\
  \hat{\omega}'(x)
  \end{array}
\right), \nonumber
\end{align}
and $\hat{\theta}(x)=\theta(x,\hat{y}(x))$, $\hat{\omega}(x)=\omega(x,\hat{y}(x))$.
By a direct calculation, we have
$$
\det(M(x))
=-\fr{\cos^2\omega}{\kappa+\sin^2\omega}(\hat{y}'(x)-\tan\alpha)(\hat{y}'(x)-\tan\beta),
$$
which, together with Lemma \ref{lem11}, gives $\det{(M(x))}\neq0$ on the curve $\hat{\ell}$. Thus system \eqref{3.27} has  a unique solution in the neighborhood of $(\tilde{x}, \tilde{y})$, that is,
$$
(\theta_{x}^u,\omega_{x}^u,\theta_{y}^u,
\omega_{y}^u)|_{\hat{\ell}}
=(\theta_{x}^\ell,\omega_{x}^\ell,
\theta_{y}^\ell,\omega_{y}^\ell)|_{\hat{\ell}}.
$$
Therefore, the boundary value problem \eqref{2.9b}, subject to  the boundary data $(\theta, \omega )|_{\widehat{BC}}=(\hat{\theta}, \hat{\omega})(x)$ and $(\theta, \omega )|_{\widehat{BA}}=(\tilde{\theta}, \tilde{\omega})(y)$, has a $C^1$ solution in $\hat{\Omega}\cup\mathcal{E_{\hat{\ell}}}$. The domain  $\hat{\Omega}\cup\mathcal{E_{\hat{\ell}}}$ is a strong determinate
domain of $\omega$. The equation $\omega(x,y)=\omega(\hat{x}-\eps,\varphi(\hat{x}-\eps))$ defines a $C^1$
curve whose graph lies in $\mathcal{E_{\hat{\ell}}}$ if $\eps>0$ is small enough in view Lemma \ref{lem11} again. The same arguments as those for $\hat{x}$ show that $\hat{x}-\eps\in \mathcal{S}$, which leads to a contradiction. Hence we obtain $\inf \mathcal{S}=x_C$ and then complete the proof of this lemma.
\end{proof}

In view of Lemma \ref{lem12}, we have the following global existence theorem.
\begin{thm}\label{thm3}
The boundary value problem \eqref{2.9a} with the boundary data $(H, \theta, \omega )|_{\widehat{CB}}=(H_0, \hat{\theta}, \hat{\omega})(x)$ and $(H, \theta, \omega )|_{\widehat{BA}}=(H_0, \tilde{\theta}, \tilde{\omega})(y)$ has a classical solution $(H_0, \theta, \omega )\in C^2$ in the domain $ABC$ with continuous sonic boundary $\widehat{AC}$.
\end{thm}

\section{Uniform regularity of solutions and sonic curve}\label{S4}

In this section, we investigate the uniform regularity of
solution up to the sonic boundary $\widehat{AC}$. Let us analyze what  we need to establish the regularity of the sonic boundary. Consider the level curves
$$
\ell^\eps=1-\sin\omega(x,y)=\eps,
$$
where $\eps$ is a small positive constant. We use \eqref{2.7} and \eqref{2.14a} to calculate
\begin{align}\label{4.1}
\ell^{\eps}_x&=-\cos\omega\omega_x=\fr{\sin\theta\cos\omega(\bar\pa^+\omega-\bar\pa^-\omega) -\cos\theta\sin\omega(\bar\pa^+\omega+\bar\pa^-\omega)}{2\sin\omega} \nonumber \\
&=(\kappa+\sin^2\omega)\bigg\{\sin\theta[\bar\pa^+\Xi-\bar\pa^-\Xi+2H_0G(\omega)] -\cos\theta\sin\omega\fr{\bar\pa^+\Xi+\bar\pa^-\Xi}{\cos\omega}\bigg\},
\end{align}
and
\begin{align}\label{4.2}
\ell^{\eps}_y&=-\cos\omega\omega_y=\fr{\cos\theta\cos\omega(\bar\pa^-\omega-\bar\pa^+\omega) -\sin\theta\sin\omega(\bar\pa^+\omega+\bar\pa^-\omega)}{2\sin\omega}\nonumber \\
&=-(\kappa+\sin^2\omega)\bigg\{\cos\theta[\bar\pa^+\Xi-\bar\pa^-\Xi+2H_0G(\omega)] +\sin\theta\sin\omega\fr{\bar\pa^+\Xi+\bar\pa^-\Xi}{\cos\omega}\bigg\}.
\end{align}
It follows that
\begin{align}\label{4.3}
(\ell^{\eps}_x)^2+(\ell^{\eps}_y)^2=(\kappa+\sin^2\omega)^2 \bigg\{[\bar\pa^+\Xi-\bar\pa^-\Xi+2H_0G(\omega)]^2+\sin^2\omega W^2\bigg\},
\end{align}
which,  together with \eqref{3.13}, leads to
\begin{align}\label{4.4}
0<(\kappa+\fr{1}{2})^2(\bar\pa^+\Xi-\bar\pa^-\Xi)^2\leq(\ell^{\eps}_x)^2+(\ell^{\eps}_y)^2\leq (\kappa+1)^2[4(M_0+H_0)^2+|W|^2],
\end{align}
where $W=(\bar\pa^+\Xi+\bar\pa^-\Xi)/\cos\omega$. Obviously, from \eqref{4.4}, we need the uniform boundedness of $W$ to establish the regularity of sonic boundary $\widehat{AC}$.

\subsection{A new partial hodograph transformation}

In order to show the uniform boundedness of $W$, we rewrite the characteristic decompositions of $\Xi$ to characterize clearly the singularity by introducing a new partial hodograph transformation.

Introduce
\begin{align}\label{4.5}
t=\cos\omega(x,y),\quad z=\theta(x,y).
\end{align}
The Jacobian of this transformation is
\begin{align}\label{4.6}
J:=\fr{\pa(t,z)}{\pa(x,y)}&=\sin\omega(\theta_x\omega_y-\theta_y\omega_x) \nonumber \\
&=-\fr{\sin\beta\bar{\pa}^+\theta-\sin\alpha\bar{\pa}^-\theta}{\sin(2\omega)} \cdot\fr{\cos\beta\bar{\pa}^+\omega-\cos\alpha\bar{\pa}^-\omega}{\sin(2\omega)}
\nonumber \\
&\qquad +\fr{\sin\beta\bar{\pa}^+\omega-\sin\alpha\bar{\pa}^-\omega}{\sin(2\omega)} \cdot\fr{\cos\beta\bar{\pa}^+\theta-\cos\alpha\bar{\pa}^-\theta}{\sin(2\omega)}
\nonumber \\
&=\fr{\bar{\pa}^+\omega\bar{\pa}^-\theta-\bar{\pa}^+\theta\bar{\pa}^-\omega}{2\cos\omega}
\nonumber \\
&=\sin\omega(\bar{\pa}^+\omega\bar{\pa}^-\Xi+\bar{\pa}^-\omega\bar{\pa}^+\Xi).
\end{align}
Inserting \eqref{2.14a} into the above gets
\begin{align}\label{4.7}
J=\fr{2F}{t}[2\bar{\pa}^+\Xi\bar{\pa}^-\Xi+H_0G(\bar{\pa}^-\Xi-\bar{\pa}^+\Xi)],
\end{align}
where
\begin{align}\label{4.8}
F=F(t)=(1-t^2)(\kappa+1-t^2),\quad G=G(t)=\bigg(\fr{1-t^2}{\kappa+1-t^2}\bigg)^{\fr{\kappa+1}{2\kappa}}.
\end{align}
In view of \eqref{3.13}, we find that $J<0$ when $t>0$ or $\omega<\fr{\pi}{2}$.

Denote
$$
U=\bar{\pa}^+\Xi,\quad V=\bar{\pa}^-\Xi.
$$
In terms of the new coordinates $(t,z)$, we have
\begin{align*}
\bar{\partial}^+&=-\fr{2F}{t}(U+H_0G)\pa_t-2t\sqrt{1-t^2}U\pa_z, \\
\bar{\partial}^-&=-\fr{2F}{t}(V-H_0G)\pa_t+2t\sqrt{1-t^2}V\pa_z.
\end{align*}
From the characteristic decomposition for $\Xi$ \eqref{2.15}, we derive the equations in terms of  $(U,V)$
\begin{align}\label{4.9}
\left\{
 \begin{array}{l}
  U_t-\fr{\sqrt{1-t^2}Vt^2}{F(V-H_0G)}U_z=-\fr{(\kappa+1)U+(\kappa+1-t^2)H_0G}{2F(V-H_0G)}\cdot\fr{U+V}{t}  +\fr{(\kappa+2-2t^2)U+(\kappa+1-t^2)H_0G}{F(V-H_0G)}Vt, \\[5pt]
  V_t+\fr{\sqrt{1-t^2}Ut^2}{F(U+H_0G)}V_z=-\fr{(\kappa+1)V-(\kappa+1-t^2)H_0G}{2F(U+H_0G)}\cdot\fr{U+V}{t}  +\fr{(\kappa+2-2t^2)V-(\kappa+1-t^2)H_0G}{F(U+H_0G)}Ut.
 \end{array}
\right.
\end{align}
Further introduce
$$
\overline{U}=U+H_0G,\quad \overline{V}=H_0G-V.
$$
Then we obtain the equations for $(\overline{U}, \overline{V})$ from \eqref{4.9}
\begin{align}\label{4.9a}
\left\{
 \begin{array}{l}
  \overline{U}_t-\fr{\sqrt{1-t^2}(\overline{V}-H_0G)t^2}{F\overline{V}}\overline{U}_z =\fr{\overline{U}}{2\overline{V}}\cdot\fr{\overline{U}-\overline{V}}{t}  +\fr{(\kappa+2-t^2)\overline{U}(\overline{U}-\overline{V})}{2F\overline{V}}t-\fr{2\kappa H_0}{(\kappa+1-t^2)^2}G^{\fr{1-\kappa}{\kappa+1}}t \\[5pt]
  \qquad \qquad \qquad \qquad \qquad \qquad -\fr{H_0G(\overline{U}-\overline{V})}{2F\overline{V}}t +\fr{[(\kappa+2-2t^2)\overline{U}-(1-t^2)H_0G](\overline{V}-H_0G)}{F\overline{V}}t,\\[5pt]
  \overline{V}_t+\fr{\sqrt{1-t^2}(\overline{U}-H_0G)t^2}{F\overline{U}}\overline{V}_z =\fr{\overline{V}}{2\overline{U}}\cdot\fr{\overline{V}-\overline{U}}{t}  +\fr{(\kappa+2-t^2)\overline{V}(\overline{V}-\overline{U})}{2F\overline{U}}t-\fr{2\kappa H_0}{(\kappa+1-t^2)^2}G^{\fr{1-\kappa}{\kappa+1}}t \\[5pt]
  \qquad \qquad \qquad \qquad \qquad \qquad -\fr{H_0G(\overline{V}-\overline{U})}{2F\overline{U}}t +\fr{[(\kappa+2-2t^2)\overline{V}-(1-t^2)H_0G](\overline{U}-H_0G)}{F\overline{U}}t.
 \end{array}
\right.
\end{align}
Furthermore, we set
$$
\widetilde{U}=\fr{1}{\overline{U}},\quad \widetilde{V}=\fr{1}{\overline{V}},
$$
which are positive and uniform bounded up to the sonic curve by Lemma \ref{lem8}. Then system \eqref{4.9a} can be rewritten as
\begin{align}\label{4.10}
\left\{
 \begin{array}{l}
  \widetilde{U}_t-\fr{\sqrt{1-t^2}(1-H_0G\widetilde{V})t^2}{F}\widetilde{U}_z =\fr{\widetilde{U}-\widetilde{V}}{2t}+F_1t, \\
  \widetilde{V}_t+\fr{\sqrt{1-t^2}(1-H_0G\widetilde{U})t^2}{F}\widetilde{V}_z =\fr{\widetilde{V}-\widetilde{U}}{2t}+F_2t,
 \end{array}
\right.
\end{align}
where
\begin{align}\label{4.10a}
\left\{
 \begin{array}{l}
  F_1=\fr{(\kappa+2-t^2)(\widetilde{U}-\widetilde{V})}{2F}+\fr{2\kappa H_0}{(\kappa+1-t^2)^2}G^{\fr{1-\kappa}{\kappa+1}}\widetilde{U}^2 -\fr{H_0G\widetilde{U}(\widetilde{U}-\widetilde{V})}{2F} \\[3pt] \qquad \quad -\fr{[(\kappa+2-2t^2)-(1-t^2)H_0G\widetilde{U}](1-H_0G\widetilde{V})\widetilde{U}}{F}, \\[3pt]
  F_2=\fr{(\kappa+2-t^2)(\widetilde{V}-\widetilde{U})}{2F}+\fr{2\kappa H_0}{(\kappa+1-t^2)^2}G^{\fr{1-\kappa}{\kappa+1}}\widetilde{V}^2 -\fr{H_0G\widetilde{V}(\widetilde{V}-\widetilde{U})}{2F} \\[3pt] \qquad \quad  -\fr{[(\kappa+2-2t^2)-(1-t^2)H_0G\widetilde{V}](1-H_0G\widetilde{U})\widetilde{V}}{F}.
 \end{array}
\right.
\end{align}
Denote
$$
\lambda_-=-\fr{\sqrt{1-t^2}(1-H_0G\widetilde{V})}{F}\cdot t^2,\qquad \lambda_+=\fr{\sqrt{1-t^2}(1-H_0G\widetilde{U})}{F}\cdot t^2,
$$
which are the eigenvalues of \eqref{4.10}. Moreover, we introduce
$$
\pa_\pm=\pa_t+\lambda_\pm\pa_z,\quad \overline{R}=\pa_+\widetilde{U}-\pa_-\widetilde{U},\quad \overline{S}=\pa_+\widetilde{V}-\pa_-\widetilde{V}.
$$
Then we obtain
\begin{align}\label{4.11}
\lambda_+-\lambda_-=\fr{\sqrt{1-t^2}[2-H_0G(\widetilde{U}+\widetilde{V})]} {F}\cdot t^2,\quad \widetilde{U}_z=\fr{\overline{R}}{\lambda_+-\lambda_-},\quad \widetilde{V}_z=\fr{\overline{S}}{\lambda_+-\lambda_-}.
\end{align}
In addition, using the commutator relation
$$
\pa_-\pa_+-\pa_+\pa_-=\fr{\pa_-\lambda_+-\pa_+\lambda_-}{\lambda_+-\lambda_-}(\pa_+-\pa_-),
$$
one obtains
\begin{align}\label{4.12}
\pa_-\overline{R}&=\pa_-\pa_+\widetilde{U}-\pa_-\pa_-\widetilde{U} \nonumber \\
&=\fr{\pa_-\lambda_+-\pa_+\lambda_-}{\lambda_+-\lambda_-}(\pa_+\widetilde{U}-\pa_-\widetilde{U}) +\pa_+\pa_-\widetilde{U}-\pa_-\pa_-\widetilde{U} \nonumber \\
&=\fr{\pa_-\lambda_+-\pa_+\lambda_-}{\lambda_+-\lambda_-}\overline{R}+ (\lambda_+-\lambda_-)(\pa_-\widetilde{U})_z,
\end{align}
and
\begin{align}\label{4.13}
\pa_+\overline{S}&=\pa_+\pa_+\widetilde{V}-\pa_+\pa_-\widetilde{V} \nonumber \\
&=\fr{\pa_-\lambda_+-\pa_+\lambda_-}{\lambda_+-\lambda_-}(\pa_+\widetilde{V}-\pa_-\widetilde{V}) +\pa_+\pa_+\widetilde{V}-\pa_-\pa_+\widetilde{V} \nonumber \\
&=\fr{\pa_-\lambda_+-\pa_+\lambda_-}{\lambda_+-\lambda_-}\overline{S}+ (\lambda_+-\lambda_-)(\pa_+\widetilde{V})_z.
\end{align}
We directly compute to obtain
\begin{align}\label{4.14}
\pa_-\lambda_+-\pa_+\lambda_- =&\bigg(\fr{\sqrt{1-t^2}t^2}{F}\bigg)_t[2-H_0G(\widetilde{U}+\widetilde{V})] \nonumber \\[3pt] &-\fr{\sqrt{1-t^2}H_0(\widetilde{U}+\widetilde{V})}{F}G_t t^2 -\fr{\sqrt{1-t^2}H_0Gt^2}{F}(\pa_-\widetilde{U}+\pa_+\widetilde{V}).
\end{align}
Thanks to the definition of $F$ in \eqref{4.8}, one gets
\begin{align}\label{4.15}
\bigg(\fr{t^2\sqrt{1-t^2}}{F}\bigg)_t=\fr{2t\sqrt{1-t^2}}{F}+\fr{t^3\sqrt{1-t^2}(\kappa+3-3t^2)}{F^2}.
\end{align}
Putting \eqref{4.15} into \eqref{4.14} and applying \eqref{4.10} one arrives at
\begin{align*}
\pa_-\lambda_+-\pa_+\lambda_- = &\bigg\{\fr{2\sqrt{1-t^2}}{F}t+ \fr{\sqrt{1-t^2}(\kappa+3-3t^2)}{F^2}t^3\bigg\}[2-H_0G(\widetilde{U}+\widetilde{V})]\\[3pt] &\ +\bigg\{ \fr{2\kappa\sqrt{1-t^2}H_0(\widetilde{U}+\widetilde{V})}{F(\kappa+1-t^2)^2}G^{\fr{1-\kappa}{\kappa+1}} -\fr{\sqrt{1-t^2}H_0G(F_1+F_2)}{F}\bigg\}t^3.
\end{align*}
Thus one has
\begin{align}\label{4.18}
\fr{\pa_-\lambda_+-\pa_+\lambda_-}{\lambda_+-\lambda_-}=\fr{2}{t}+th,
\end{align}
where $h$ denotes
\begin{align*}
h=\fr{\kappa+3-3t^2}{F}+\fr{1}{2-H_0G(\widetilde{U}+\widetilde{V})}\bigg\{\fr{2\kappa H_0(\widetilde{U}+\widetilde{V})}{(\kappa+1-t^2)^2}G^{\fr{1-\kappa}{\kappa+1}}-H_0G(F_1+F_2)\bigg\}.
\end{align*}
Furthermore, it follows from \eqref{4.10} that
\begin{align}\label{4.19}
(\pa_-\widetilde{U})_z=\fr{\widetilde{U}_z-\widetilde{V}_z}{2t}+tg_1\widetilde{U}_z+tg_2\widetilde{V}_z,
\end{align}
where $g_1$ and $g_2$ denote
\begin{align*}
g_1=&\fr{(\kappa+2-t^2)-H_0G(2\widetilde{U}-\widetilde{V})}{2F}+\fr{4\kappa H_0\widetilde{U}}{(\kappa+1-t^2)^2}G^{\fr{1-\kappa}{\kappa+1}} \\[3pt] &\ +\fr{[2(1-t^2)H_0G\widetilde{U}-(\kappa+2-2t^2)](1-H_0G\widetilde{V})}{F}, \\[3pt]
g_2=&\fr{H_0G\widetilde{U}-(\kappa+2-t^2)}{2F} +\fr{[(\kappa+2-2t^2)-(1-t^2)H_0G\widetilde{U}]H_0G\widetilde{U}}{F}.
\end{align*}
Similarly, $\widetilde{V}$ satisfies
\begin{align}\label{4.20}
(\pa_+\widetilde{V})_z=\fr{\widetilde{V}_z-\widetilde{U}_z}{2t}+th_1\widetilde{V}_z+th_2\widetilde{U}_z,
\end{align}
where $h_1$ and $h_2$ denote
\begin{align*}
h_1=&\fr{(\kappa+2-t^2)-H_0G(2\widetilde{V}-\widetilde{U})}{2F}+\fr{4\kappa H_0\widetilde{V}}{(\kappa+1-t^2)^2}G^{\fr{1-\kappa}{\kappa+1}} \\[3pt] &\ +\fr{[2(1-t^2)H_0G\widetilde{V}-(\kappa+2-2t^2)](1-H_0G\widetilde{U})}{F}, \\[3pt]
h_2=&\fr{H_0G\widetilde{V}-(\kappa+2-t^2)}{2F} +\fr{[(\kappa+2-2t^2)-(1-t^2)H_0G\widetilde{V}]H_0G\widetilde{V}}{F}.
\end{align*}
We insert \eqref{4.19} and \eqref{4.20} into \eqref{4.12} and \eqref{4.13}, respectively, and then use \eqref{4.18} and \eqref{4.11} to obtain
\begin{align}\label{4.21}
\left\{
 \begin{array}{l}
  \pa_-\overline{R}=\bigg(\fr{5}{2t}+t(h+g_1)\bigg)\overline{R}+\bigg(tg_2-\fr{1}{2t}\bigg)\overline{S}, \\[10pt]
  \pa_+\overline{S}=\bigg(\fr{5}{2t}+t(h+h_1)\bigg)\overline{S}+\bigg(th_2-\fr{1}{2t}\bigg)\overline{R}.
 \end{array}
\right.
\end{align}
Here the functions $h, g_1, g_2, h_1$ and $h_2$ are uniformly bounded up to $t=0$ by their explicit expressions.

To treat the current problem, we further introduce
\begin{align}\label{4.22}
\widetilde{R}=\fr{\overline{R}}{t^2},\quad \widetilde{S}=\fr{\overline{S}}{t^2}.
\end{align}
Then \eqref{4.21} can be transformed as
\begin{align}\label{4.23}
\left\{
 \begin{array}{l}
  \pa_-\bigg(t^{-\fr{1}{2}}\widetilde{R}\bigg)=t^{-\fr{3}{2}} \bigg(t^2(h+g_1)\widetilde{R}+\fr{tg_2-1}{2}\widetilde{S}\bigg),    \\[10pt]
  \pa_+\bigg(t^{-\fr{1}{2}}\widetilde{S}\bigg)=t^{-\fr{3}{2}} \bigg(t^2(h+h_1)\widetilde{S}+\fr{th_2-1}{2}\widetilde{R}\bigg).
 \end{array}
\right.
\end{align}

\subsection{Regularity in partial hodograph plane}

We employ \eqref{4.23} to derive the regularity of solutions near the line $t =0$.

We use $A'B'C'$ to denote the region $ABC$ in the $(t,z)$ plane. Let $(\bar{z},0)$ be any fixed point on the degenerate segment $\widehat{A'C'}$ and $t_{P}$ be small such that the point $P(t_P, \bar{z})$ stays in the domain $A'B'C'$. From the point $P$, we draw $\lambda_+$ and $\lambda_-$-characteristics, called $z_+(P)$ and $z_-(P)$, up to the segment $\widehat{A'C'}$ at $P_1$ and $P_2$, respectively. According to the the uniform boundedness of $h, g_1, g_2, h_1$ and $h_2$ near $t=0$, then for any constant $\nu\in(0,1]$, we can choose $t_P<1$ small enough such that
\begin{align}\label{4.24}
t^2|h+g_1|\leq\fr{\nu}{4}, \quad t^2|h+h_1|\leq\fr{\nu}{4}, \quad |tg_2-1|\leq 1+\nu, \quad |th_2-1|\leq 1+\nu
\end{align}
hold in the whole domain $PP_1P_2$. Moreover, we draw the $\lambda_\pm$-characteristics from the point $(\bar{z},0)$ to encircle a domain $D_\nu(\bar{z})$ with $\widehat{PP_1}$ and $\widehat{PP_2}$.
For any $(z,t)\in D_\nu(\bar{z})$, we denote by $a(z_a,t_a)$ and $b(z_b,t_b)$ the intersection points of the $\lambda_-$ and $\lambda_+$-characteristics through $(z, t)$ with the boundaries $\widehat{PP_1}$ and $\widehat{PP_2}$, respectively.  See Fig. \ref{fig4}.

Let
$$
\widehat{K}=\max\bigg\{\max_{D_\nu(\bar{z})}|\widetilde{R}(z_a,t_a)|+1,\ \ \max_{D_\nu(\bar{z})}|\widetilde{S}(z_b, t_b)|+1\bigg\},
$$
which is well-defined and uniformly bounded in the domain $D_\nu(\bar{z})$. We have the following lemma.

\begin{figure}[htbp]
\begin{center}
\includegraphics[scale=0.55]{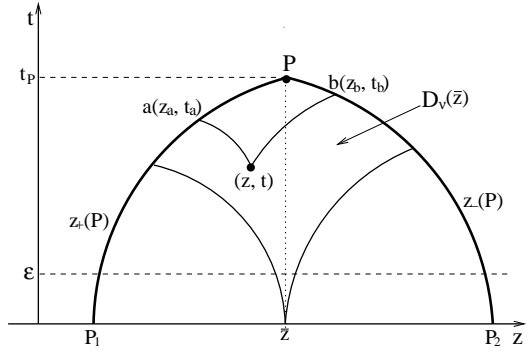}
\caption{\footnotesize The region $D_\nu(\bar{z})$.}\label{fig4}
\end{center}
\end{figure}

\begin{lem}\label{lem13}
Let $(\bar{z},0)$ be any point on the degenerate line $\widehat{A'C'}$ and $\nu\in(0,1]$ be any constant. Then there holds
\begin{align}\label{4.25}
|t^\nu \widetilde{R}|\leq \widetilde{K},\quad |t^\nu \widetilde{S}|\leq \widetilde{K}, \qquad \forall\ (z,t)\in D_\nu(\bar{z})
\end{align}
for some uniform positive constant $\widetilde{K}$.
\end{lem}
\begin{proof}
We use the bootstrap argument employed in \cite{SWZ} to show the lemma. For a fixed $\eps\in(0,t_P)$, we denote
$$
\widehat{D}_\eps:=\{(z,t):\ \eps\leq t\leq t_P, z_-(P)\leq z\leq z_+(P)\}\cap D_\nu(\bar{z})
$$
and $\widetilde{K}_\eps=\max\limits_{\widehat{D}_\eps}\{|t^\nu \widetilde{R}|, |t^\nu \widetilde{S}|\}$. If for any $0<\eps<t_P$, one has $\widetilde{K}_\eps\leq \widehat{K}$.  Then \eqref{4.25} holds. Otherwise, we assume there exists $\eps_0\in(0,t_P)$ such that $\widetilde{K}_{\eps_0}>\widehat{K}$.

Let $z_{-}^a(t)$ and $z_{+}^b(t)$ be the $\lambda_-$ and $\lambda_+$-characteristics  through the point $(z_{\eps_0},\eps_0)\in D_\nu(\bar{z})$ and intersecting $\widehat{PP_1}$ and $\widehat{PP_2}$ at points $a(z_a,t_a)$ and $b(z_b,t_b)$, respectively. We integrate the first equation of \eqref{4.23} along the curve $z_{-}^a(t)$ from $t(\geq\eps_0)$ to $t_a$,  and apply \eqref{4.24} and the definition of $\widehat{K}$ to find
\begin{align*}
&\bigg|\fr{\widetilde{R}}{t^{\fr{1}{2}}}\bigg|=\bigg|\fr{\widetilde{R}(z_a,t_a)}{t_{a}^{\fr{1}{2}}} +\int_{t}^{t_a}\fr{\tau^2(h+g_1)\widetilde{R} +\fr{tg_2-1}{2}\widetilde{S}}{\tau^{\fr{3}{2}}}(z_{-}^a(\tau),\tau)\ {\rm d}\tau\bigg|  \\[3pt]
\leq&\bigg|\fr{\widetilde{R}(z_a,t_a)}{t_{a}^{\fr{1}{2}}} \bigg| +\int_{t}^{t_a}\fr{\fr{\nu}{4}|\widetilde{R}|+\fr{1+\nu}{2}|\widetilde{S}|}{\tau^{\fr{3}{2}}}(z_{-}^a(\tau),\tau)\ {\rm d}\tau \leq\bigg|\fr{\widetilde{R}(z_a,t_a)}{t_{a}^{\fr{1}{2}}} \bigg| +\fr{2+3\nu}{4}\widetilde{K}_{\eps_0}\int_{t}^{t_a}\fr{1}{\tau^{\fr{3}{2}+\nu}}\ {\rm d}\tau  \\[3pt]
=&\bigg|\fr{\widetilde{R}(z_a,t_a)}{t_{a}^{\fr{1}{2}}} \bigg|
+\fr{2+3\nu}{4}\widetilde{K}_{\eps_0}\fr{1}{\fr{1}{2}+\nu}\bigg(\fr{1}{t^{\fr{1}{2}+\nu}} -\fr{1}{t_{b}^{\fr{1}{2}+\nu}}\bigg)<\fr{(2+3\nu)\widetilde{K}_{\eps_0}}{2+4\nu}t^{-\fr{1}{2}-\nu} <\widetilde{K}_{\eps_0}t^{-\fr{1}{2}-\nu},
\end{align*}
which leads to a strict inequality
\begin{align}\label{4.26}
|\widetilde{R}|_{t=\eps_0}<\widetilde{K}_{\eps_0}\eps_{0}^{-\nu},
\end{align}
on the line segment $\{t=\varepsilon_0\}\cap\widehat{D}_{\eps_0}$. One takes  similar arguments for the second equation \eqref{4.23} to arrive at
\begin{align}\label{4.27}
|\widetilde{S}|_{t=\eps_0}<\widetilde{K}_{\eps_0}\eps_{0}^{-\nu}.
\end{align}
We combine \eqref{4.26} and \eqref{4.27} to conclude that the maximum values of $|t^\nu\widetilde{R}|$ and $|t^\nu\widetilde{S}|$ in the domain $\widehat{D}_{\eps_0}$ can only attain on $\eps_0<t\leq t_P$, which
also holds in a larger domain $\widehat{D}_{\eps'}$,  $\eps'<\eps_0$. Then one may repeat the above processes to extend the domain larger and larger and until the whole domain $D_\nu(\bar{z})$. The proof  is complete.
\end{proof}

The results in Lemma \ref{lem13} can be extended a small interval on the segment $\widehat{A'C'}$. Let
$(\bar z, 0)$ be any point in $\widehat{A'C'}$ and $\mu$ be a small constant such that $(\bar z-\mu, \bar z+\mu)\subset\widehat{A'C'}$. From the point $Q_1:=(\bar z-\mu, 0)$ ($Q_2:=(\bar z+\mu, 0)$), we draw the  $\lambda_-$ (reap. $\lambda_+$)-characteristic up to the boundary $\widehat{PP_1}$ (resp. $\widehat{PP_2}$) at a point $Q^{*}_1$ (reap. $Q^{*}_2$).
Denote $D_\nu({\bar z}_\mu)$ the domain bounded by the boundaries $\widehat{PQ^{*}_1}, \widehat{Q^{*}_1Q_1}, \widehat{Q_1Q_2}, \widehat{Q_2Q^{*}_2}$ and $\widehat{Q^{*}_2P}$. Then we have

\begin{lem}\label{lem14}
Let $(\bar{z},0)$ be any point on the degenerate line $\widehat{A'C'}$ and $\mu$ be a small positive constant such that $(\bar z-\mu, \bar z+\mu)\subset\widehat{A'C'}$. Then, for any constant $\nu\in(0,1]$, there exists a positive constant $\widetilde{K}$ depending on the interval $(\bar z-\mu, \bar z+\mu)$ such that there hold
\begin{align}\label{4.28}
|t^\nu\widetilde{R}|\leq \widetilde{K},\quad |t^\nu\widetilde{S}|\leq \widetilde{K}, \qquad \forall\ (z,t)\in D_\nu({\bar z}_\mu).
\end{align}
\end{lem}

Now we establish the uniform boundedness of the function $\widetilde{W}:=(\widetilde{U}-\widetilde{V})/t$.
\begin{lem}\label{lem15}
The function $\widetilde{W}$ is uniformly bounded up to the sonic boundary $\widehat{A'C'}$.
\end{lem}
\begin{proof}
We derive the equation of $\widetilde{W}$ by \eqref{4.10}, \eqref{4.11} and \eqref{4.22}
\begin{align}\label{4.29}
\widetilde{W}_t=&(F_1-F_2)+\fr{\sqrt{1-t^2}t}{F}[(1-H_0G\widetilde{V})\widetilde{U}_z +(1-H_0G\widetilde{U})\widetilde{V}_z] \nonumber \\
=&(F_1-F_2) +\fr{(1-H_0G\widetilde{V})(t^\nu\widetilde{R}) +(1-H_0G\widetilde{U})(t^\nu\widetilde{S})}{2-H_0G(\widetilde{U}+\widetilde{\widetilde{V}})}t^{1-\nu}.
\end{align}
By choosing $t$ small enough and $\nu=1/2$ in Lemma \ref{lem14}, we know that  the two terms in the right hand side of  \eqref{4.29} are uniformly bounded,  which implies the uniform boundedness of $\widetilde{W}$.
\end{proof}
Thanks to Lemma \ref{lem15} and the definitions of $\widetilde{U}$ and $\widetilde{V}$, we acquire the uniform boundedness of $W:=(U+V)/t$ in the whole domain $A'B'C'$.  That is, there exists a uniform constant $K^*$ such that
\begin{align}\label{4.30}
\bigg|\fr{U+V}{t}(z,t)\bigg|\leq K^* \quad \forall\ (z,t)\in A'B'C',
\end{align}
which, together with \eqref{4.9},  implies  that $|\pa_-U|$ and $|\pa_+V|$ are uniformly bounded in the whole domain $A'B'C'$. Furthermore, we use \eqref{4.10} and \eqref{4.11} again to obtain
\begin{align*}
\widetilde{U}_t&=\fr{1-H_0G\widetilde{V}}{2-H_0G(\widetilde{U}+\widetilde{V})}(t^2\widetilde{R}) + \fr{1}{2}\widetilde{W}+F_1t,\\
\widetilde{V}_t&=-\fr{1-H_0G\widetilde{U}}{2-H_0G(\widetilde{U}+\widetilde{V})}(t^2\widetilde{S}) - \fr{1}{2}\widetilde{W}+F_2t,
\end{align*}
from which we see that  $(|\widetilde{U}_t|, |\widetilde{V}_t|)$ and of course  $(|U_t|, |V_t|)$ are uniformly bounded in the whole domain $A'B'C'$ containing the degenerate line $t=0$. Hence, the two functions $\bar\pa^+\Xi$ and $-\bar\pa^-\Xi$ approach a common value on the degenerate curve $\widehat{AC}$ with at least a rate of $\cos\omega$.

Next we determine the uniform regularity of $U(z,t)$, $V(z,t)$ and $W(z,t)$ up to $\widehat{A'C'}$.
\begin{lem}\label{lem16}
The functions $U(z,t)$, $V(z,t)$ and $W(z,t)$ are uniformly $C^{\fr{1}{3}}$ continuous in the whole domain $A'B'C'$, including the degenerate line $\widehat{A'C'}$.
\end{lem}
\begin{proof}
Suppose that $(z_1,0)$ and $(z_2,0)$ $(z_1<z_2)$ are any two points on the degenerate curve $\widehat{A'C'}$. Let $(z_m,t_m)$ be the intersection point of the $\lambda_+$-characteristic $z^+$ and $\lambda_-$-characteristic $z^-$ starting from $(z_1,0)$ and $(z_2,0)$ respectively. Recalling the eigenvalues of \eqref{4.9} yields
\begin{align*}
\left\{
 \begin{array}{l}
  \fr{{\rm d}}{{\rm d}t}z^-=-\fr{\sqrt{1-t^2}V}{F(V-H_0G)}t^2, \\
  z^-|_{t=0}=z_2,
 \end{array}
\right.\quad
\left\{
 \begin{array}{l}
  \fr{{\rm d}}{{\rm d}t}z^+=\fr{\sqrt{1-t^2}U}{F(U+H_0G)}t^2, \\
  z^+|_{t=0}=z_1,
 \end{array}
\right.
\end{align*}
from which we get, using \eqref{3.13}
\begin{align}\label{4.31}
\underline{K}t_m\leq|z_m-z_i|^{\fr{1}{3}}\leq\overline{K}t_m,\quad i=1,2,
\end{align}
for some positive constants $\underline{K}$ and $\overline{K}$. In view of the uniform boundedness of $|\pa_-U|$ and $|\pa_+V|$, we integrate $\pa_+V$ from $(z_1,0)$ to $(z_m,t_m)$ and $\pa_-U$ from $(z_2,0)$ to $(z_m,t_m)$ and employ \eqref{4.31} to acquire
\begin{align*}
|V(z_1,0)-V(z_m,t_m)|\leq K_1t_m, \quad |U(z_2,0)-U(z_m,t_m)|\leq K_1t_m
\end{align*}
for some uniform constant $K_1$. Thus we use the fact $|(U+V)(z,t)|\leq K^*t$ and \eqref{4.31} to obtain
\begin{align*}
&|U(z_2,0)-U(z_1,0)|=|U(z_2,0)+V(z_1,0)| \\
\leq \ &|U(z_2,0)-U(z_m,t_m)|+|U(z_m,t_m)+V(z_m,t_m)| +|V(z_m,t_m)-V(z_1,0)| \\
\leq \ &(2K_1+K^*)t_m\leq K_2|z_2-z_1|^{\fr{1}{3}}=K_2|(z_2,0)-(z_1,0)|^{\fr{1}{3}},
\end{align*}
where $K_2=(2K_1+K^*)/\underline{K}$ is a uniformly constant.

For any two points $(z_1,t_1)$ and $(z_2,t_2)$ $(z_1\leq z_2, 0\leq t_1\leq t_2)$ in the region $A'B'C'$,  one arrives at
\begin{align*}
|U(z_2,t_2)-U(z_1,t_1)|=&|U(z_1,t_2)-U(z_1,t_1)| \\
\leq& \max_{A'B'C'}|U_t|\cdot|t_2-t_1|=\max_{A'B'C'}|U_t|\cdot|(z_2,t_2)-(z_1,t_1)|.
\end{align*}
 if $z_1=z_2$, by using the uniform boundedness of $U_t$.

If $z_1<z_2$, the analysis consists of  two cases:

\noindent Case I. $t_1\geq(z_2-z_1)$. For this case, we note the relation
$$
U_z=-\fr{F(U+H_0G)^2\widetilde{R}}{\sqrt{1-t^2}[2-H_0G(\widetilde{U}+\widetilde{V})]}
$$
derived from \eqref{4.11} and \eqref{4.22}. Then we take $\nu=1/2$ in \eqref{4.28} to get
\begin{align*}
|U(z_2,t_2)-U(z_1,t_1)|&\leq |U(z_2,t_2)-U(z_2,t_1)|+|U(z_2,t_1)-U(z_1,t_1)| \\
&\leq \max_{A'B'C'}|U_t|\cdot|t_2-t_1|+|U_z|\cdot|z_2-z_1| \\
&\leq\max_{A'B'C'}|U_t|\cdot|t_2-t_1|+Kt_{1}^{-\fr{1}{2}}\cdot|z_2-z_1|
\\
&\leq\max_{A'B'C'}|U_t|\cdot|t_2-t_1|+K\cdot|z_2-z_1|^{\fr{1}{2}}\leq K|(z_2,t_2)-(z_1,t_1)|^{\fr{1}{3}}
\end{align*}
for some uniformly constant $K$.

\noindent Case II. $t_1<(z_2-z_1)$. There exists a uniform constant $K>0$ such that
\begin{align*}
&|U(z_2,t_2)-U(z_1,t_1)|\\
\leq& |U(z_2,t_2)-U(z_2,t_1)|+|U(z_2,t_1)-U(z_2,0)|+|U(z_2,0)-U(z_1,0)|+|U(z_1,0)-U(z_1,t_1)| \\
\leq & \max_{A'B'C'}|U_t|\cdot|t_2-t_1|+\max_{A'B'C'}|U_t|\cdot t_1+K_2|z_2-z_1|^{\fr{1}{3}}+\max_{A'B'C'}|U_t|\cdot t_1
\\
\leq & 2\max_{A'B'C'}|U_t|(|t_2-t_1|+|z_2-z_1|)+K_2|z_2-z_1|^{\fr{1}{3}} \leq K|(z_2,t_2)-(z_1,t_1)|^{\fr{1}{3}}.
\end{align*}
Therefore, the function $U(z,t)$ is uniformly $C^{\fr{1}{3}}$ continuous in the whole domain $A'B'C'$.
Analogously, we can show the uniformly $C^{\fr{1}{3}}$ continuity of $V$ and $W$.
\end{proof}

\subsection{Regularity in the  original physical plane}

In this subsection, we establish the uniform regularity of solutions and the regularity of the sonic curve $\widehat{AC}$ in the $(x,y)$ plane to complete the proof of Theorem \ref{thm1}. The analysis is divided into
three steps.

\textbf{Step I.} This step is devoted to verifying that the map $(x,y)\mapsto(z,t)$ is an one-to-one mapping. We use the contradiction argument. If the map is not one-to-one, one assumes $(\hat{x},\hat{y})$ and $(\tilde{x},\tilde{y})$ are two distinct points in the domain $ABC$ such that $\cos\omega(\hat{x},\hat{y})=\cos\omega(\tilde{x},\tilde{y})$ and $\theta(\hat{x},\hat{y})=\theta(\tilde{x},\tilde{y})$. Thus the two points $(\hat{x},\hat{y})$ and $(\tilde{x},\tilde{y})$ are on a level curve $\ell^\eps(x,y)=1-\sin\omega(x,y)=\eps\geq0$. On the other hand, we claim that the function $\theta$ is strictly monotonic along each level curve $\ell^\eps(x,y)=\eps\geq0$. In fact, we directly compute by \eqref{3.18a} and \eqref{4.1}-\eqref{4.2}
\begin{align*}
&(\theta_x, \theta_y)\cdot(\ell^{\eps}_y, -\ell^{\eps}_x) \\
=&(\kappa+\sin^2\omega) \bigg\{-(\sin\beta\bar\pa^+\Xi+\sin\alpha\bar\pa^-\Xi)[\cos\theta(\bar\pa^+\Xi-\bar\pa^-\Xi+2H_0G) +\sin\theta\sin\omega W] \\
&\qquad \qquad \qquad \ \ +(\cos\beta\bar\pa^+\Xi+\cos\alpha\bar\pa^-\Xi)[\sin\theta(\bar\pa^+\Xi-\bar\pa^-\Xi+2H_0G) -\cos\theta\sin\omega W]\bigg\} \\[3pt]
=&(\kappa+\sin^2\omega)\bigg\{(\bar\pa^+\Xi-\bar\pa^-\Xi+2H_0G)\sin\omega(\bar\pa^+\Xi-\bar\pa^-\Xi) -\sin\omega\cos\omega W(\bar\pa^+\Xi+\bar\pa^-\Xi)\bigg\} \\
=&\sin\omega(\kappa+\sin^2\omega)[2H_0G(\bar\pa^+\Xi-\bar\pa^-\Xi)-4\bar\pa^+\Xi\bar\pa^-\Xi]>0.
\end{align*}
Hence we get a contradiction with the assumption $\theta(\hat{x},\hat{y})=\theta(\tilde{x},\tilde{y})$.

\textbf{Step II.} We assert that the function $\omega(x,y)$ is uniformly $C^{\fr{1}{2}}$-continuous in the whole domain $ABC$, containing the sonic boundary $\widehat{AC}$. To prove the assertion, we rewrite \eqref{2.7} as
\begin{align*}
\pa_x=\fr{\cos\theta}{2}\cdot\fr{\bar\pa^++\bar\pa^-}{\cos\omega} -\fr{\sin\theta}{2\sin\omega}(\bar\pa^+-\bar\pa^-), \quad \pa_y=\fr{\sin\theta}{2}\cdot\fr{\bar\pa^++\bar\pa^-}{\cos\omega} +\fr{\cos\theta}{2\sin\omega}(\bar\pa^+-\bar\pa^-),
\end{align*}
which combined with \eqref{2.14a} yields
\begin{align}\label{4.32}
\begin{array}{l}
\cos\omega\pa_x\omega=\sin\omega(\kappa+\sin^2\omega)[\cos\theta W-\sin\theta(\bar\pa^+\Xi-\bar\pa^-\Xi+2H_0G)], \\
\cos\omega\pa_y\omega=\sin\omega(\kappa+\sin^2\omega)[\sin\theta W+\cos\theta(\bar\pa^+\Xi-\bar\pa^-\Xi+2H_0G)].
\end{array}
\end{align}
These, together with the uniform boundedness of $W$, imply
\begin{align}\label{4.33}
|\cos\omega\omega_x|+|\cos\omega\omega_y|\leq K
\end{align}
for a uniform positive constant $K$. Thus one gets
$$
|(\fr{\pi}{2}-\omega)\omega_x|+ |(\fr{\pi}{2}-\omega)\omega_y|\leq\fr{\fr{\pi}{2}-\omega}{\sin(\fr{\pi}{2}-\omega)}K\leq 2K,
$$
which indicates that the function $(\pi/2-\omega)^2$ is uniformly Lipschitz continuous in terms of $(x,y)$.  That is, for any two points $(x',y')$ and $(x'',y'')$ in the whole domain $ABC$, there holds
\begin{align}\label{4.34}
\bigg|\bigg(\fr{\pi}{2}-\omega(x',y')\bigg)^2 -\bigg(\fr{\pi}{2}-\omega(x'',y'')\bigg)^2\bigg|\leq2K|(x',y')-(x'',y'')|.
\end{align}
Moreover, we use the fact $(\pi/2-\omega)\geq0$ to arrive at
\begin{align*}
|\omega(x',y')-\omega(x'',y'')|^2 =&\bigg|\bigg(\fr{\pi}{2}-\omega(x',y')\bigg)-\bigg(\fr{\pi}{2}-\omega(x'',y'')\bigg)\bigg|^2 \\[5pt]
\leq &\bigg|\bigg(\fr{\pi}{2}-\omega(x',y')\bigg)^2-\bigg(\fr{\pi}{2}-\omega(x'',y'')\bigg)^2\bigg|,
\end{align*}
which along with \eqref{4.34} leads to
\begin{align}\label{4.35}
|\omega(x',y')-\omega(x'',y'')|\leq\sqrt{2K}|(x',y')-(x'',y'')|^\fr{1}{2}.
\end{align}
Thus $\omega(x,y)$ is uniformly $C^{\fr{1}{2}}$-continuous in the whole domain $ABC$.

\textbf{Step III.} We claim that, for any function $\phi(z,t)\in C^{\fr{1}{3}}$ defined in the whole domain $A'B'C'$, the function $\widetilde{\phi}(x,y):=\phi(z,t)$ is uniformly $C^{\fr{1}{6}}$-continuous in the whole domain $ABC$ in terms of $(x,y)$. Let $(x',y')$ and $(x'',y'')$ be any two points in $ABC$ and $(z',t')$ and $(z'',t'')$ be two points in $A'B'C'$ such that $t'=\cos\omega(x',y')$, $z'=\theta(x',y')$ and $t''=\cos\omega(x'',y'')$, $z''=\theta(x'',y'')$. Due to assumption $\phi(z,t)\in C^{\fr{1}{3}}$, we have
\begin{align}\label{4.36}
&|\widetilde{\phi}(x'',y'')-\widetilde{\phi}(x',y')|=|\phi(z'',t'')-\phi(z',t')|\leq K|(z'',t'')-(z',t')|^{\fr{1}{3}} \nonumber \\
=&K\bigg\{\bigg(\cos\omega(x'',y'')-\cos\omega(x',y')\bigg)^2 +\bigg(\theta(x'',y'')-\theta(x',y')\bigg)^2\bigg\}^{\fr{1}{6}}
\end{align}
for some uniform constant $K$. Furthermore, making use of \eqref{4.35} sees
\begin{align*}
|\cos\omega(x'',y'')-\cos\omega(x',y')| &\leq\bigg|2\sin\fr{\omega(x'',y'')-\omega(x',y')}{2}\bigg| \\
&
\leq|\omega(x'',y'')-\omega(x',y')|\leq \sqrt{2K}|(x'',y'')-(x',y')|^{\fr{1}{2}}.
\end{align*}
We put the above into \eqref{4.1} and apply \eqref{3.19} to obtain
$$
|\widetilde{\phi}(x'',y'')-\widetilde{\phi}(x',y')|\leq K|(x'',y'')-(x',y')|^{\fr{1}{6}},
$$
which means that $\widetilde{\phi}(x,y)$ is uniformly $C^{\fr{1}{6}}$-continuous in the whole domain $ABC$.

Combining with Lemma \ref{lem16} and the above result, we achieve that the functions $\bar\pa^+\Xi(x, y)$,  $\bar\pa^-\Xi(x, y)$ and $W(x,y)$ are uniformly $C^{\fr{1}{6}}$-continuous in terms of $(x, y)$ in the whole domain $ABC$. Then we recall \eqref{3.18a} and \eqref{4.32} that $\theta(x,y)$ and $\sin\omega(x,y)$ are uniformly $C^{1,\fr{1}{6}}$ in the whole domain $ABC$. To complete the proof of Theorem \ref{thm1}, it remains to show the $C^{1,\fr{1}{6}}$-continuity of the sonic curve $\widehat{AC}$. From the definitions \eqref{4.1} and \eqref{4.2}, the functions $\ell^{\eps}_x(x,y)$ and $\ell^{\eps}_y(x,y)$ are uniformly $C^{\fr{1}{6}}$-continuous in the whole domain $ABC$. In addition, we combine with \eqref{3.13}, \eqref{4.4} and \eqref{4.30} to get
$$
0<\tilde{m}\leq[\ell^{\eps}_x(x,y)]^2+[\ell^{\eps}_y(x,y)]^2\leq \widetilde{M}
$$
for some uniform constants $\tilde{m}$ and $\widetilde{M}$. Thus the sonic curve $\widehat{AC}$ is $C^{1,\fr{1}{6}}$ continuous and the proof of Theorem \ref{thm1} is completed.

\section*{Acknowledgements}

The first author is supported by NSF of Zhejiang Province (no. LY17A010019). The second author is supported by NSFC (nos. 11771054, 91852207) and Foundation of LCP.


\end{document}